\pdfoutput=1 
\documentclass[12pt,cd]{amsart}%
\usepackage{color}
\usepackage{amsmath}
\usepackage{amsfonts}
\usepackage{amssymb}
\usepackage{graphicx}
\usepackage{amssymb}
\providecommand{\U}[1]{\protect\rule{.1in}{.1in}}
\theoremstyle{plain}
\newtheorem{theo+}{Theorem}[section]
\newtheorem{prop+}[theo+]{Proposition}
\newtheorem{coro+}[theo+]{Corollary}
\newtheorem{lemm+} [theo+]{Lemma}
\newtheorem{deep+}  [theo+]  {Deep Result}
\newtheorem{fact+}  [theo+]  {Fact}
\theoremstyle{definition}
\newtheorem{exam+}  [theo+]  {Example}
\newtheorem{rema+}  [theo+]  {Remark}
\newtheorem{defi+}  [theo+]  {Definition}
\newtheorem{xca+}[theo+]{Exercise}
\newenvironment{theorem}{\begin{theo+}}{\end{theo+}}
\newenvironment{proposition}{\begin{prop+}}{\end{prop+}}
\newenvironment{corollary}{\begin{coro+}}{\end{coro+}}
\newenvironment{lemma}{\begin{lemm+}}{\end{lemm+}}

\newenvironment{remark}{\begin{rema+}}{\end{rema+}}
\newenvironment{definition}{\begin{defi+}}{\end{defi+}}

\numberwithin{equation}{section}

\def\Norm#1_#2{\Vert#1\Vert_{#2}}

\def\1{{\bf 1}}

\begin{document}
\title[Positivity of Shimura operators]{Positivity of Shimura operators}
\author{Siddhartha Sahi}
\email{\textit{siddhartha.sahi@gmail.com }}
\address{Department of Mathematics, Rutgers University, New Brunswick, NJ, USA }
\author{Genkai Zhang}
\email{\textit{genkai@chalmers.se}}
\address
{Mathematical Sciences, Chalmers University of Technology and Mathematical
Sciences, G\"oteborg University, SE-412 96 G\"oteborg, Sweden.}
\address{Korea Institute for Advanced Study, Seoul, Korea}

\thanks{Research by G. Zhang partially supported by the Swedish Science Council (VR)}

\begin{abstract}
In \cite{Shimura-1990} G. Shimura introduced a family of invariant
differential operators that play a key role in the study of nearly holomorphic
automorphic forms, and he asked for a determination of their \textquotedblleft
domain of positivity\textquotedblright. In this paper we relate the
eigenvalues of Shimura operators to certain polynomials introduced by A.
Okounkov, which leads to an explicit answer to Shimura's questions.

\end{abstract}
\maketitle

\section{Introduction}

In this paper we answer an old question of G. Shimura on the spectrum of
certain invariant differential operators on a Hermitian symmetric space. These
operators were first introduced by Shimura in \cite{Shimura-1990} for
classical Hermitian symmetric spaces, and they play a key role in his higher
rank generalization of the theory of nearly holomorphic automorphic forms. In
order to describe Shimura's question, and our answer, it is convenient to
introduce the following notation for \textquotedblleft
partitions\textquotedblright:
\begin{equation}
\Lambda=\left\{  \lambda\in\mathbb{Z}^{n}\mid\lambda_{1}\geq\lambda_{2}%
\geq\cdots\geq\lambda_{n}\geq0\right\}  \text{,\quad}\left\vert \lambda
\right\vert =\lambda_{1}+\cdots+\lambda_{n}. \label{=Lambda}%
\end{equation}
We will write $1^{j}$ for the partition $\left(  1,\ldots,1,0,\ldots,0\right)
$ with $j$ \textquotedblleft ones.\textquotedblright

Now suppose $G/K$ is an irreducible Hermitian symmetric space of rank $n$. Let
$\mathfrak{g}$ and $\mathfrak{k}$ denote the complexified Lie algebras of $G$
and $K$, and let
\[
\mathfrak{g}=\mathfrak{k}+\mathfrak{p=k}+\mathfrak{p}^{+}+\mathfrak{p}^{-}
\]
be the corresponding Cartan decomposition. Let $\mathfrak{t}$ be a Cartan
subalgebra of $\mathfrak{k}$, then $\mathfrak{t}$ is also a Cartan subalgebra
of $\mathfrak{g}$; furthermore there is a distinguished family of strongly
orthogonal roots for $\mathfrak{t}$ in $\mathfrak{p}^{+}$%
\[
\left\{  \gamma_{1},\ldots,\gamma_{n}\right\}  \subseteq\Delta\left(
\mathfrak{t,p}^{+}\right)
\]
called the Harish-Chandra roots.

Now $\mathfrak{p}^{+}$ and $\mathfrak{p}^{-}$ are abelian Lie algebras, which
are contragredient as $K$-modules. Let $W_{\lambda} $ be the $K$-module with
highest weight $\sum_{i}\lambda_{i}\gamma_{i}$ and let $W_{\lambda}^*$
be its
contragradient, then by a result of Schmid we have $K$-module isomorphisms
\[
U\left(  \mathfrak{p}^{+}\right)  \approx S\left(  \mathfrak{p}^{+}\right)
\approx\oplus_{\lambda\in\Lambda}W_{\lambda},\quad U\left(  \mathfrak{p}%
^{-}\right)  \approx S\left(  p^{-}\right)  \approx\oplus_{\lambda\in\Lambda
}W_{\lambda}^{\ast}
\]
Let $u_{\lambda}$ denote the image of $1\in End\left(  W_{\lambda}\right)  $
under the sequence of maps
\[
End\left(  W_{\lambda}\right)  \approx W_{\lambda}^{\ast}\otimes W_{\lambda
}\hookrightarrow U\left(  \mathfrak{p}^{-}\right)  \otimes U\left(
\mathfrak{p}^{+}\right)  \overset{mult}{\longrightarrow}U\left(
\mathfrak{g}\right).
\]
Then $u_{\lambda}$ belongs to $U\left(  \mathfrak{g}\right)  ^{K}$ and
its right action on $G$ descends to
an operator $\mathcal{L}_{\lambda}\in\mathbf{D}\left(  G/K\right)  $. In fact $\left\{  \mathcal{L}_{\lambda}:\lambda\in
\Lambda\right\}  $ is a linear basis and $\left\{  \mathcal{L}_{1^{j}%
}:j=1,\ldots,n\right\}  $ is an independent generating set for $\mathbf{D}%
\left(  G/K\right)  $. These are the Shimura operators.

The algebra $\mathbf{D}\left(  G/K\right)  $ is commutative and its
eigenfunctions are the spherical functions $\Phi_{x}$. These are parametrized
by the set $\mathfrak{a}^{\ast}/W_{0}$, where $\mathfrak{a\subseteq p}$ is a
Cartan subspace and $W_{0}=W\left(  \mathfrak{a,g}\right)  $ is the restricted
Weyl group. More precisely the parameter $x\in\mathfrak{a}^{\ast}$ defines an
irreducible spherical subquotient $J\left(  x\right)  $ of a minimal principal
series representation of $G$ and $\Phi_{x}$ is its spherical matrix
coefficient. We write $\mathcal{U}$ for the set of spherical unitary
parameters%
\[
\mathcal{U}=\left\{  x\in\mathfrak{a}^{\ast}\mid J\left(  x\right)  \text{ is
unitarizable}\right\}  .
\]
The determination of $\mathcal{U}$ is an important problem, which is as yet
unsolved in complete generality. In this connection, the Shimura operators
have the following positivity property (see Proposition 5.1 below). Let
$c_{\lambda,x}$ denote the eigenvalue of the modified operator $\mathcal{L}%
_{\lambda}^{\prime}=\left(  -1\right)  ^{\left\vert \lambda\right\vert
}\mathcal{L}_{\lambda}$ on $\Phi_{x}$, then we have
\begin{equation}
c_{\lambda,x}\geq0\text{ for all }x\in\mathcal{U}. \label{c-pos}%
\end{equation}

Motivated in part by (\ref{c-pos}), Shimura asked for a determination of the
sets
\begin{align}
\mathcal{A}  &  =\left\{  x:c_{\lambda,x}\geq0\text{ for all }\lambda\right\}
\label{=A}\\
\mathcal{G}  &  =\left\{  x:c_{1^{j},x}\geq0\text{ for all }j\right\}
\label{=G}%
\end{align}
Evidently we have $\mathcal{U}\subseteq\mathcal{G}\subseteq\mathcal{A}$ but these
sets are quite different in general. In this paper we give an explicit formula
for $c_{\lambda,x}$, thereby answering Shimura's question. To describe our
answer we need some further notation. First by a classical result of
Harish-Chandra we have
\[
D\Phi_{x}=\eta_{D}\left(  x\right)  \Phi_{x}%
\]
where $\eta_{D}=\eta\left(  D\right)  $ is the image of $D$ under the
Harish-Chandra homomorphism%
\begin{equation}
\eta:\mathbf{D}\left(  G/K\right)  \rightarrow S\left(  \mathfrak{a}\right)
^{W_{0}}\approx P\left(  \mathfrak{a}^{\ast}\right)  ^{W_{0}}\text{.}
\label{=HC-hom}%
\end{equation}
Thus we have $c_{\lambda,x}=c_{\lambda}\left(  x\right)  $, where
\[
c_{\lambda}=\eta\left(  \mathcal{L}_{\lambda}^{\prime}\right)  =\left(
-1\right)  ^{\left\vert \lambda\right\vert }\eta\left(  \mathcal{L}_{\lambda
}\right)
\]
and so the determination of the sets (\ref{=A}, \ref{=G}) reduces to the
determination of $\eta\left(  \mathcal{L}_{\lambda}\right)  $.

Our first result relates the Shimura operators to certain polynomials
$P_{\lambda}\left(  x;\tau,\alpha\right)  $. These polynomials, denoted
$P_{\lambda}^{ip}\left(  x;\tau,\alpha\right)  $ in \cite{Koor}, are the
$q\rightarrow1$ limits of a family of poynomials introduced by A. Okounkov
\cite[Definition 1.1]{Okounkov}, and they generalize an analogous family
defined by one of the authors and studied together with F. Knop
\cite{Knop-Sahi}. To describe the $P_{\lambda}$ it is convenient to define
$\delta=\left(  n-1,\ldots,1,0\right)  $ and to set
\begin{equation}
\rho=\rho_{\tau,\alpha}=\left(  \rho_{1},\ldots,\rho_{n}\right)  ,\quad
\rho_{i}=\tau\delta_{i}+\alpha=\tau\left(  n-i\right)  +\alpha. \label{=rta}%
\end{equation}
The polynomial $P_{\lambda}\left(  x;\tau,\alpha\right)  $ has total degree
$2\left\vert \lambda\right\vert $ in $x_{1},\ldots,x_{n}$, its coefficients
are rational functions in two parameters $\tau$ and $\alpha$, it is even and
symmetric, i.e. invariant under all permutations and sign changes of the
$x_{i}$, and among all such polynomials it is characterized up to scalar
multiple by its vanishing at points of the form%
\[
\left\{  \mu+\rho:\mu\in\Lambda,\left\vert \mu\right\vert \leq\left\vert
\lambda\right\vert ,\mu\neq\lambda\right\}.
\]
For generic $\tau$ the set $\left\{  P_{\lambda};\lambda\in\Lambda\right\}  $
is a linear basis of the space $\mathcal{Q}$ of even symmetric polynomials.

Now suppose $G/K$ is a Hermitian symmetric space as before. Then the
restricted root system $\Sigma\left(  \mathfrak{a},\mathfrak{g}\right)  $ is
of type $BC_{n}$, with (potentially) three root lengths, and we fix a choice
of positive roots. The positive long roots have multiplicity $1$ and
constitute a basis of $\mathfrak{a}^{\ast}$, thus we may use them to identify
$\mathfrak{a}^{\ast}$ with $\mathbb{C}^{n}$. This identifies $W_{0}$ with the
group of all permutation and sign changes, and $P\left(  \mathfrak{a}^{\ast
}\right)  ^{W_{0}}$ with the algebra $\mathcal{Q}$. Moreover if we denote the
multiplicities of short and medium roots by $2b$ and $d$ respectively, then
the half-sum of positive roots $\Sigma\left(  \mathfrak{a},\mathfrak{g}%
\right)  $ is given by $\rho=\rho_{\tau,\alpha}$ as in (\ref{=rta}) where%
\begin{equation}
\tau=d/2,\quad\alpha=\left(  b+1\right)  /2. \label{=tau-alpha}%
\end{equation}

\begin{theorem}
Let $G/K$ be a Hermitian symmetric space with $\tau,\alpha$ as in
(\ref{=tau-alpha}), then
\[
\eta\left(  \mathcal{L}_{\lambda}\right)  =k_{\lambda}P_{\lambda}\left(
x;\tau,\alpha\right)\]
where $k_{\lambda}=k_{\lambda}\left(  \tau,\alpha\right)  $ is an explicit
\emph{positive} constant described in (\ref{=k-mu}).
\end{theorem}

This is proved more generally for line bundles on $G/K$ in Theorem \ref{thm1}
below. In view of this, we introduce the signed versions of the Okounkov
polynomials
\begin{equation}
q_{\lambda}\left(  x\right) : =q_{\lambda}\left(  x;\tau,\alpha\right)
=\left(  -1\right)  ^{\left\vert \lambda\right\vert }P_{\lambda}\left(
x;\tau,\alpha\right). \label{=q}%
\end{equation}

\begin{corollary}
The Shimura sets are given explicitly as follows:%
\begin{align*}
\mathcal{A}  &  =\left\{  x:q_{\lambda}\left(  x\right)  \geq0\text{ for all
}\lambda\right\}  ,\\
\mathcal{G}  &  =\left\{  x:q_{1^{j}}\left(  x\right)  \geq0\text{ for all
}j\right\}.
\end{align*}

\end{corollary}

This is  Corollary \ref{cor1-2} below.

Since one has explicit formula for $P_{\lambda}$, and hence $q_{\lambda}$,
recalled in (\ref{koor-fo}) below, this gives a complete characterization of
the Shimura sets. In particular, we obtain the following explicit description
of $\mathcal{G}$. If $I$ is a subset of $\left\{  1,\ldots,n\right\}  $ with
$j$ elements $i_{1}<\cdots<i_{j}$, then we define
\[
\varphi_{I}\left(  x\right)  =\prod\nolimits_{k=1}^{j}\left[  \left(
\rho_{i_{k}+j-k}\right)  ^{2}-x_{i_{k}}^{2}\right]  ,\quad\varphi_{j}\left(
x\right)  =\sum\nolimits_{\left\vert I\right\vert =j}\varphi_{I}\left(
x\right).
\]

\begin{theorem}
We have $q_{1^{j}}=\varphi_{j}$ for all $j$, and hence%
\[
\mathcal{G}=\left\{  x:\varphi_{j}\left(  x\right)  \geq0\text{ for all
}j\right\}.
\]

\end{theorem}

This is Theorem \ref{G-set} below.

The description of $\mathcal{A}$ involves infinitely many polynomial
inequalities $q_{\lambda}\geq0$, and it is natural to ask whether
$\mathcal{A}$ can in fact be described by a \emph{finite} set of inequalities.
While we do not know the answer to this question in general, we show below
that this is indeed the case for the real points in $\mathcal{A}$ for the rank
$2$ groups $U\left(  m,2\right)  $. However for $m>2$ the characterization
involves \emph{non-polynomials} in an essential way. This is in sharp contrast
with the unitary parameter set $\mathcal{U}$, whose description involves only
\emph{linear} functions (see Remark 5.9 below).

We give two independent derivations of this result. The first depends on a
formula for $P_{\lambda}\left(  x;\tau,\alpha\right)  $ for $n=2$ as a
hypergeometric polynomial \cite{Koor}. By symmetry it suffices to describe the
sets
\[
\mathcal{A}_{0}=\mathcal{A\cap C}\text{,}\quad\mathcal{G}_{0}=G\cap
\mathcal{C}\text{,\quad}\mathcal{C}=\left\{  x\in\mathbb{R}^{n}:x_{1}%
\geq\cdots\geq x_{n}\geq0\right\}.
\]
For $G=$ $U\left(  m,2\right)  $ we have
\[
\left(  \rho_{1},\rho_{2}\right)  =\left(  \alpha+\tau,\alpha\right)  =\left(
\dfrac{m+1}{2},\dfrac{m-1}{2}\right)
\]
and we introduce two triangular regions as below%
\begin{align*}
T_{1}  &  =\left\{  x\mid\rho_{2}\geq x_{1}\geq x_{2}\geq0\right\}  ,\\
T_{2}  &  =\left\{  x\mid x_{1}\geq x_{2}\geq\rho_{2},\quad x_{1}+x_{2}%
\leq\rho_{1}+\rho_{2}\right\}.
\end{align*}
We also write $\left(  a\right)_{k}=a\left(  a+1\right)  \cdots\left(
a+k-1\right)  $ for the Pochammer symbol of product
of increasing factors and 
$\left(  a\right)_{k}^-=a\left(  a-1\right)  \cdots\left(
a-k+1\right)$ for the decreasing factors,
 and define %
\[
R\left(  x_1, x_2\right)  =\sum_{k=0}^{\infty}\frac{\left(  \rho_{2}%
+x_{2}\right)  _{k}\left(  \rho_{2}-x_{2}\right)  _{k}}{\left(  \rho_{1}%
+x_{1}\right)  _{k}\left(  \rho_{1}-x_{1}\right)  _{k}}.
\]

\begin{theorem}
\label{Th:Um2}For $G=U(m,2)$ we have

\begin{enumerate}
\item $\mathcal{G}_{0}=T_{1}\cup T_{2}$.
\item $\mathcal{A}_{0}=T_{1}\cup W$ where $W=\left\{  x\in T_{2}\mid R\left(
x\right)  \geq0\right\} $.
\end{enumerate}
\end{theorem}

This is proved in Theorems \ref{thm4.4} and \ref{pos-dom}.

In section \ref{sec-Alt} we give an alternative description of $W$ in terms of
the functions
\[
s\left(  t\right)  =s_{m}\left(  t\right)  =\frac{\sin\pi t}{\left(
t+1\right)  _{m}}, \, 
S\left(x_1, x_2\right)  =
\frac{
s\left(  x_{1}\right)  -s\left(  x_{2}\right)
}{x_{1}-x_{2}}.
\]

\begin{theorem}
\label{W2}The set $W$ of Theorem \ref{Th:Um2} can also be described as
follows:%
\[
W=\left\{  x\in T_{2}\mid
S(x_1-\alpha, x_2-\alpha)\geq0\right\}.
\]

\end{theorem}

This description of $W$ is facilitated by a Weyl-type formula for the Okounkov
polynomials $P_{\lambda}\left(  x;\tau,\alpha\right)  $ for $\tau=1$, that we
describe below. It tunes out that for $\tau=1$, the $P_{\lambda}$ can be
expressed in terms of rank $1$ Okounkov polynomials, which are given
explicitly as follows
\[
p_{l}\left(  z;\alpha\right)  =\prod\nolimits_{k=0}^{l-1}\left[  z^{2}-\left(
k+\alpha\right)  ^{2}\right].
\]
For $\mu\in\Lambda$ we define the \emph{alternant} $a_{\mu}$ to be the
determinant of the $n\times n$ matrix
\[
\left[  p_{\mu_{j}}\left(  x_{i};\alpha\right)  \right]  _{i,j=1}^{n}.
\]
For $\delta=\left(  n-1,\ldots,1,0\right)  $, the alternant is in fact the
Vandermonde determinant%
\[
a_{\delta}\left(  x;\alpha\right)  =\prod\nolimits_{i<j}\left(  x_{i}%
^{2}-x_{j}^{2}\right).
\]

\begin{theorem}
For $\tau=1$ we have $P_{\lambda}\left(  x;\tau,\alpha\right)  =\dfrac
{a_{\lambda+\delta}\left(  x;\alpha\right)  }{a_{\delta}\left(  x\right)  }.$
\end{theorem}

This is proved in Theorem \ref{thm-b} below.

\begin{corollary}
For $G=U\left(  m,n\right)  $ we have $\left(  \tau,\alpha\right)  =\left(
1,\frac{m-n+1}{2}\right)  $ and
\[
\mathcal{A}=\left\{  x:\left(  -1\right)  ^{\left\vert \lambda\right\vert
}\dfrac{a_{\lambda+\delta}\left(  x;\alpha\right)  }{a_{\delta}\left(
x;\alpha\right)  }\geq0\text{ for all }\lambda\right\}  .
\]

\end{corollary}

This description of $\mathcal{A}$, although still infinite, is much more
explicit. It plays a key role in the proof of Theorem \ref{W2}, which involves
a study of the limiting behavior of $a_{\lambda+\delta}\left(  x;\alpha
\right)  .$

\emph{Acknowledgement.} We  thank Tom Koornwinder for helpful
correspondence. Part of this work was done when Genkai Zhang
was visiting KIAS, Korea as a KIAS Scholar in April 2016. He would
like to thank the institute for its support and hospitality.
We also thank Alejandro Ginory for computational assistance, which was extremely helpful at an early stage of this project.

\section{Preliminaries}

\label{sect-1}

We shall introduce the Shimura operators following \cite{Shimura-1990}. See
also \cite{gz-shimura} for further study and references therein.

\subsection{Lie algebras of Hermitian type}

We will denote real Lie algebras by $\mathfrak{g}_{0}$, $\mathfrak{k}_{0}$
etc. and denote their complexifications by $\mathfrak{g,k}$ etc. Let $\left(
\mathfrak{g}_{0},\mathfrak{k}_{0}\right)  $ be an irreducible Hermitian
symmetric pair of real rank $n$, and let
\[
\mathfrak{g}=\mathfrak{p}^{-}+\mathfrak{k}+\mathfrak{p}^{+}%
\]
be its Harish-Chandra decomposition into $\left(  -1,0,1\right)  $ eigenspaces
with respect to a suitable central element $Z\in\mathfrak{k}$. Let
$\mathfrak{t}$ be a Cartan subalgebra of $\mathfrak{k}$, then $\mathfrak{t}$
is also a Cartan subalgebra of $\mathfrak{g},$ and we fix a compatible choice
of positive root systems satisfying
\[
\Delta^{+}\left(  \mathfrak{g,t}\right)  =\Delta^{+}\left(  \mathfrak{k,t}%
\right)  \cup\Delta\left(  \mathfrak{p}^{+}\mathfrak{,t}\right)
\]
Let $\gamma_{1},\ldots,\gamma_{n}\in\Delta\left(  \mathfrak{p}^{+}%
\mathfrak{,t}\right)  $ be the Harish-Chandra strongly orthogonal roots, and
let $h_{j}\in\mathfrak{t}$ be coroot corresponding to $\gamma_{j}$. Then we
have a commuting family of $sl_{2}$-triples%
\[
\left\{  h_{j},e_{j}^{+},e_{j}^{-}\right\}  ,\quad e_{j}^{\pm}\in
\mathfrak{p}^{\pm}.%
\]
We fix an invariant bilinear form on $\mathfrak{g}$ such that
\begin{equation}
(e_{1}^{+},e_{1}^{-})=1 \label{eq:norm}%
\end{equation}
Let $\mathfrak{t}_{-}=%
{\textstyle\sum\nolimits_{j=1}^{n}}
\mathbb{C}h_{j}$ be the span of the $h_{j},$ and let $\mathfrak{t}_{+}$ be the
orthogonal comlement of $\mathfrak{t}_{-}$ in $\mathfrak{t}$; then we have an
orthogonal decomposition $\mathfrak{t}=\mathfrak{t}_{-}+\mathfrak{t}_{+}$.We
also define%
\[
e_{j}=e_{j}^{-}+e_{j}^{+},\quad\mathfrak{a}=%
{\textstyle\sum\nolimits_{j=1}^{n}}
\mathbb{C}e_{j},\quad\mathfrak{h=a+t}_{+}.
\]
Then $\mathfrak{a}$ is a maximal abelian subspace of $\mathfrak{p}$ and
$\mathfrak{h}$ is a maximally split Cartan subalgebra of
$\mathfrak{\mathfrak{g}}$. The restricted root system $\Sigma\left(
\mathfrak{a},\mathfrak{\mathfrak{g}}\right)  $ is of type $BC_{n}$; more
precisely, if $\{\varepsilon_{j}\}\subset\mathfrak{a}^{\ast}$ is the basis
dual to $\{e_{j}\}\subset$ $\mathfrak{a}$, then we have
\[
\Sigma\left(  \mathfrak{a},\mathfrak{\mathfrak{g}}\right)  =\left\{
\pm\varepsilon_{i},\pm\varepsilon_{i}\pm\varepsilon_{j},\pm2\varepsilon
_{i}\right\} .
\]
The long roots $\pm2\varepsilon_{i}$ have multiplicity $1$, and they are
conjugate to $\pm\gamma_{j}$ via the Cayley transform that carries
$\mathfrak{a}$ to $\mathfrak{t}_{-}$ and $\mathfrak{h}$ to $\mathfrak{t}$
\cite{FK}. We denote the multiplicity of the medium roots $\pm\varepsilon
_{i}\pm\varepsilon_{j}$ and the short roots $\pm\varepsilon_{i}$ by $d$ and
$2b$ respectively, and we fix the following choice of positive roots
\[
\Sigma^{+}\left(  \mathfrak{a},\mathfrak{\mathfrak{g}}\right)  =\left\{
\varepsilon_{i}\right\}  \cup\left\{  \varepsilon_{i}\pm\varepsilon_{j}\mid
i<j\right\}  \cup\left\{  2\varepsilon_{i}\right\}  \text{.}%
\]
Then the half sum of positive roots is given as follows
\[
{\rho=\rho\left(  \mathfrak{a},\mathfrak{\mathfrak{g}}\right)  }=\sum2\rho
_{i}\varepsilon_{i}\text{,\quad}\rho_{i}=\frac{1}{2}\left[  d{(n-i)+1+b}%
\right]  .
\]

Let $G$ be the connected Lie group of the adjoint group of $\mathfrak{g}$ with
Lie algebras $\mathfrak{g}_{0}$, and $K$ the corresponding subgroup with Lie
algebra $\mathfrak{k}_{0}$. Then $G/K$ is a non-compact Hermitian symmetric space.



\subsection{Hua-Schmid decomposition}

\label{hua-schmid} As before we define
\begin{equation}
\Lambda=\left\{  \lambda\in\mathbb{Z}^{n}:\lambda_{1}\geq\cdots\geq\lambda
_{n}\geq0\right\}  . \label{Lam-set}%
\end{equation}
and for $\lambda\in\Lambda$ we let $W_{\lambda}$ denote the irreducible
$K$-module with highest weight $\sum_{j}\lambda_{j}\gamma_{j}$.

We write $S\left(  V\right)  $ for the symmetric algebra of a vector space and
$U\left(  \mathfrak{s}\right)  $ for the enveloping algebra of a Lie algebra.
We note that $\mathfrak{\mathfrak{p}^{-}\approx(\mathfrak{p}^{+})^{\ast}}$ as
$K$-modules, and also that since $\mathfrak{\mathfrak{p}^{\pm}}$ are abelian
we have $U\left(  \mathfrak{\mathfrak{p}^{\pm}}\right)  \approx S\left(
\mathfrak{\mathfrak{p}^{\pm}}\right)  $. By a result of Schmid (see e.g.
\cite{FK}) we have multiplicity free $K$-module decompositions
\[
U\left(  \mathfrak{\mathfrak{p}^{+}}\right)  \approx\oplus_{\lambda\in\Lambda
}W_{\lambda},\qquad U\left(  \mathfrak{\mathfrak{p}^{-}}\right)  \approx
\oplus_{\lambda\in\Lambda}W_{\lambda}^{\ast}%
\]

The space of all holomorphic polynomials on $\mathfrak{p}^{+}$ is naturally
identified with $S\left(  \mathfrak{\mathfrak{p}^{-}}\right)  $. It is
naturally equipped with the Fock space norm \cite{FK}, with the inner product
on $\mathfrak{p}^{-}\subset\mathfrak{g}$ being normalized above. We denote the
corresponding reproducing kernel for the subspace $W_{\lambda}^{\ast}$ by
$K_{\lambda}(z, w)$, $z, w\in\mathfrak{p}^{+}$.

\subsection{Line bundles over $G/K$}

\label{sec-2.2} In the subsequent discussion we will need to study equivariant
line bundles on $G/K$. Such bundles correspond to multiplicative characters of
$K$, and one has the following standard result.

\begin{lemma}
There exists a unique character $\iota$ of $K$ whose differential restricts to
$\frac{1}{2}\left(  \gamma_{1}+\cdots+\gamma_{n}\right)  $ on $\mathfrak{t}%
_{-}.$
\end{lemma}

\begin{proof}
See \cite{schlicht}.
\end{proof}

\begin{remark}
In geometric terms, the character $\iota$ is a generator of the Picard group
of holomorphic line bundles on $G/K$. We have
\begin{equation}
\iota^{p_{0}}\left(  k\right)  =\operatorname{det}(\text{ ad}%
(k)|_{\mathfrak{p}^{+}})\text{,\quad}p_{0}=2+(n-1)d+b\text{,} \label{=p0root}%
\end{equation}
and thus $\iota$ is the $p_{0}$-th root of the canonical line bundle on $G/K$.
The differential $d\iota$ vanishes on $\mathfrak{t}_{+}$ iff $G/K$ is tube
type, otherwise $d\iota|_{\mathfrak{t}_{+}}$ is as in \cite[Section 5, p.
288]{schlicht} (with $l=1$ there).
\end{remark}

For any integer $p$, we consider the character $\iota^{p}(k)=\iota(k)^{p}$ and
we write $C^{\infty}(G/K,p)$ for the space of smooth sections of the
corresponding holomorphic line bundle over $G/K$. Explicitly we have
\[
C^{\infty}(G/K,p)=\{f\in C^{\infty}(G);f(gk)=\iota(k)^{p}f(g),g\in G,k\in
K\}.
\]
The group $G$ acts on $C^{\infty}(G/K,p)$ via the left regular action.

\begin{remark}
In \cite{gz-shimura} the eigenvalues of the Shimura operators on $G/K$ were
studied by explicit computations. The line bundle parameter $p$ here
corresponds to $-\nu$ there. When $-p=\nu>p_{0}-1$ there is a holomorphic
discrete series representation in the space $L^{2}(G/K,p)$, and it is in the
common kernels of the "adjoint" Shimura operators $\mathcal{M}_{\mu}$ in
(\ref{l-m-mu}) below. Similar results hold for the non-compact dual $U/K$ of
$G/K$.
\end{remark}

\subsection{The Schlichtkrull-Cartan-Helgason theorem}

Finite dimensional representations of $\mathfrak{k}$ and $\mathfrak{g}$ are
parametrized by their highest weights as linear functionals on $\mathfrak{t}$.
The representations we shall treat in the present paper have their highest
weights being determined by their restriction to $\mathfrak{t}_{-}$, so if a
linear function on $\mathfrak{t}$ is of the form $\sum_{j=1}^{n}\lambda
_{j}\gamma_{j}$ on $\mathfrak{t}_{-}$ and is dominant with respect to the
roots $\Delta^{+}(\mathfrak{k},\mathfrak{t})$ respectively $\Delta
^{+}(\mathfrak{g},\mathfrak{t})$, then the corresponding representations of
$\mathfrak{k}$ and $\mathfrak{g}$ will be denoted by $W_{\lambda}$
respectively $V_{\lambda}$.

By the Cartan-Helgason theorem the finite dimensional $K$-spherical
representations of $G$ can be parameterized in terms of their highest weights
restricted to Cartan subspace $\mathfrak{a}$, they are precisely of the form
$\sum2\lambda_{i}\varepsilon_{i} $, where $\lambda$ ranges over the same set
(\ref{Lam-set}).

We shall need a generalization by Schlichtkrull of the Cartan-Helgason theory
for the line bundle case. We call a vector $v\in V$ in a representation $(\pi,
V)$ of $G$ a $(K, \iota^{p})$-spherical vector if
\[
\pi(k) v=\iota(k)^{p}v
\]
Then one has the following result \cite[Th 7.2]{schlicht}.

\begin{lemma}
\label{schli} Let $p$ be an integer. For each $\lambda\in\Lambda$ there is a
unique representation $V_{\lambda, p} $ of $G$ in $C^{\infty}(G/K,p)$ whose
highest weight restricts to $\sum_{j=1}^{n}\left(\lambda_{j}+\frac{|p|}2\right)\gamma
_{j}$ on $\mathfrak{t_{-}}$. In particular each space $V_{\lambda, p}$
contains a unique $(K, \iota^{\pm p})$-spherical vector $v_{\pm p}$ up to
non-zero scalars.
\end{lemma}

We denote $W_{\lambda, p}=W_{\lambda}\otimes C_{\frac p2}$, which is an
irreducible representation of $\mathfrak{k}$. Notice that the highest weight
of $V_{\lambda, p}$ is the same as $W_{\lambda, |p|}$. Also, $V_{\lambda, p}$
contains both $(K, \iota^{p})$ and $(K, \iota^{-p})$ and spherical vectors.
(This is not true for infinite dimensional highest weight representations.) It
will be convenient to treat the space $V_{\lambda, p} $ as an abstract
representation, more precisely if $V_{\lambda, p}$ is any irreducible
representation space of $U $ with highest weight $\sum_{j}\lambda_{j}%
\gamma_{j}$ and containing a $(K, \iota^{-p})$-spherical normalized vector
$v_{-p}$, then the map
\[
v\in V_{\lambda, p}\to f_{v}(g)=\left(\pi_{\lambda}(g^{-1})v, v_{-p}\right)
\]
is a realization of $V_{\lambda, p}$ in the space $C^{\infty}(G/K, \iota^{p}%
)$. Here $\left(\cdot, \cdot\right)$ is the Hermitian inner product in $V_{\lambda, p}$.


\subsection{ Shimura operators}

Shimura operators are parametrized by the set $\Lambda$. More precisely for
each $\mu\in\Lambda$ the Shimura operator corresponds to the identity element
$1\in Hom\left(  W_{\mu},W_{\mu}\right)  $ via the multiplication in the
universal enveloping algebra $U\left(  \mathfrak{g}\right)  $%
\[
1\in End\left(  W_{\lambda}\right)  \approx W_{\lambda}^{\ast}\otimes
W_{\lambda}\hookrightarrow U\left(  \mathfrak{p}^{-}\right)  \otimes U\left(
\mathfrak{p}^{+}\right)  \overset{mult}{\longrightarrow}U\left(
\mathfrak{g}\right)  .
\]
Explicitly let $\{\xi_{\alpha}\}$ be a basis of $W_{\mu}\subset
S(\mathfrak{p}^{+})$ and $\{\eta_{\alpha}\}$ be the dual basis $W_{\mu}^{\ast
}\subset S(\mathfrak{p}^{-})$ and define%

\begin{equation}
\mathcal{L}_{\mu}=\sum\nolimits_{\alpha}\eta_{\alpha}\xi_{\alpha}%
,\quad\mathcal{M}_{\mu}=\sum\nolimits_{\alpha}\xi_{\alpha}\eta_{\alpha},
\label{l-m-mu}%
\end{equation}
viewed as elements of $U(\mathfrak{g})^{K}$ acting in $C^{\infty}(G)$ (or
$C^{\infty}(U)$ for the compact dual $U$ of $G$)as left invariant differential operator.

Alternatively we may take an orthonormal basis $\xi_{\alpha}$ of $W_{\mu
}\subset S(\mathfrak{p}^{+})$and the dual basis is then given by
$\eta_{\alpha}=\xi_{\alpha}^{\ast}$, where $v\rightarrow v^{\ast}$ is the
conjugation in $\mathfrak{p}$ with respect to the real form $\mathfrak{p}_{0}$
extended to $S(\mathfrak{p})$. Thus we have%

\begin{equation}
\mathcal{L}_{\mu}=\sum\nolimits_{\alpha}\xi_{\alpha}^{\ast}\xi_{\alpha}%
,\quad\mathcal{M}_{\mu}=\sum\nolimits_{\alpha}\xi_{\alpha}\xi_{\alpha}^{\ast}.
\label{L-mu}%
\end{equation}
Note that the operators $(-1)^{\mu}\mathcal{L}_{\mu}$ and $(-1)^{\mu
}\mathcal{M}_{\mu}$ descend to $G$-invariant differential operator on
$C^{\infty}(G/K,p)$ and are formally non-negative on $C^{\infty}(G/K,p)$; they
also define $U$-invariant differential operators on $C^{\infty}(U/K,p)$ for
the compact symmetric space $U/K$ with $\mathcal{L}_{\mu}$ and $\mathcal{M}%
_{\mu}$ being non-negative instead. See Section \ref{section-pro} below.

For $\mu\in\Lambda$, the Shimura operator $\mathcal{L}_{\mu}$ and
$\mathcal{M}_{\mu}$ have order $2\left\vert \mu\right\vert $ where%
\[
\left\vert \mu\right\vert =\mu_{1}+\cdots+\mu_{n}.
\]

The  Harish-Chandra homomorphism
for invariant differential operators on $C^\infty(G/K)$
can be generalized
to line bundles over $G/K$; see e.g. \cite{Shimura-1990,
Shimeno-jfa} and references therein. More precisely there exists
a Weyl group invariant polynomials $\eta_p(\mathcal L_{\mu})$,
the Harish-Chandra homomorphism of 
the Shimura operator $\mathcal L_{\mu}$, such that
$\mathcal{L}_{\mu}$ on the irreducible representations $V_{\lambda, p}\subset
C^{\infty}(G/K, p)$ above by 
\begin{equation}
  \label{eq:eta-p}
\mathcal{L}_{\mu} 
v
=\eta_p\left( \mathcal{L}_{\mu}\right)\left( \lambda+\frac p2+\rho\right)
  v, \,
\quad v\in V_{\lambda, p}.  
\end{equation}
Furthermore  $\eta_p(\mathcal L_{\mu}) \in\mathcal{Q}_{2|\mu|}$.

\section{Okounkov polynomials\label{oko}}

In this section we discuss some key properties of a family of polynomials
introduced by Okounkov \cite{Okounkov}, or rather their $q\rightarrow1$ limit
as discussed \cite[(7.2)]{Koor}. These polynomials play an important role in
the theory of symmetric functions, and they generalize an earlier family of
polynomials introduced by one of us in \cite{sahi-spec}, and studied in
\cite{Knop-Sahi, Sahi-interp}.

Let $\mathbb{F}=\mathbb{Q}(\tau,\alpha)$ be the field of rational functions in
$\tau,\alpha$. Consider the polynomial ring $\mathbb{F}\left[  x_{1}%
,\ldots,x_{n}\right]  $ equipped with the natural action of the group
$W=S_{n}\ltimes\left(  \mathbb{Z}/2\right)  ^{n}$ by permutations and sign
changes, and let%
\[
\mathcal{Q=}\mathbb{F}\left[  x_{1},\ldots,x_{n}\right]  ^{W}%
\]
be the subring of even symmetric polynomials. Okounkov polynomials
$P_{\lambda}\left(  x;\tau,\alpha\right)  $ form a distinguished linear basis
of $\mathcal{Q}$, indexed by the set $\Lambda$. We refer the reader to
\cite{Okounkov} and \cite[(7.2)]{Koor} for more background on these
polynomials, noting that in the latter paper they are referred to as
Okounkov's $BC_{n}$ type interpolation polynomials, and denoted $P_{\lambda
}^{ip}\left(  x;\tau,\alpha\right)  .$

\subsection{Combinatorial formula}

We first recall
\cite{Koor, Macdonald-book} some basic combinatorial terminology associated to partitions.
The length of a partition $\lambda\in\Lambda$ is
\[
l\left(  \lambda\right)  =\max\left\{  i\mid\lambda_{i}>0\right\}  .
\]
The Young diagram of $\lambda$ is the collection of \textquotedblleft
boxes\textquotedblright\ $s=(i,j)$
\[
\left\{  \left(  i,j\right)  :1\leq i\leq l\left(  \lambda\right)  ,1\leq
j\leq\lambda_{i}\right\}  .
\]
The \emph{arm/leg/coarm/coleg} of $s=\left(  i,j\right)  \in\lambda$ are
defined as follows:%
\[
a\left(  s\right)  =\lambda_{i}-j,l\left(  s\right)  =\#\left\{
k>i\mid\lambda_{k}\geq j\right\}  ,a^{\prime}(s)=j-1,l^{\prime}(s)=i-1.
\]
We write $\mu\subseteq\lambda$ if $\mu_{i}\leq\lambda_{i}$ for all $i$. In this
case the diagram of $\mu$ can be regarded as a subset of $\lambda$, and we
define%
\begin{gather*}
\left(  R\backslash C\right)  _{\lambda\backslash\mu}=\left\{  s\in\mu\mid
a_{\lambda}\left(  s\right)  >a_{\mu}\left(  s\right)  ,l_{\lambda}\left(
s\right)  =l_{\mu}\left(  s\right)  \right\}  ,\\
\psi_{\lambda\backslash\mu}=%
{\displaystyle\prod}
{}_{s\in\left(  R\backslash C\right)  _{\lambda\backslash\mu}}\frac{b_{\mu
}\left(  s\right)  }{b_{\lambda}\left(  s\right)  },\quad b_{\lambda}\left(
s\right)  =\frac{\tau l_{\lambda}\left(  s\right)  +a_{\lambda}\left(
s\right)  +\tau}{\tau l_{\lambda}\left(  s\right)  +a_{\lambda}\left(
s\right)  +1}.
\end{gather*}
A \emph{reverse} tableau $T$ of shape $\lambda$ is an assignment of the boxes
$s\in\lambda$ with numbers $T(s)\in\{1,\cdots,n\}$ so that $T(i,j)$ is
strongly decreasing in $i$ and weakly decreasing in $j$. Such a tableau
defines a sequence of partitions
\[
0=\lambda^{\left(  n\right)  }\subset\cdots\subset\lambda^{\left(  1\right)
}\subset\lambda^{\left(  0\right)  }=\lambda,\quad\lambda^{\left(  i\right)
}=\left\{  s\mid T\left(  s\right)  >i\right\}
\]
and we set%
\[
\psi_{T}=\prod\nolimits_{i=1}^{n}\psi_{\lambda^{\left(  i-1\right)
}\backslash\lambda^{\left(  i\right)  }}.
\]

\begin{definition}
The Okounkov polynomial is
\begin{equation}
P_{\lambda}\left(x;\tau,\alpha\right)=\sum_{T}\psi_{T}\prod_{s\in\lambda}\left[
x_{T(s)}^{2}-(a_{\lambda}^{\prime}(s)+\tau\left(  n-T(s)-l_{\lambda}^{\prime
}(s)\right)  +\alpha)^{2}\right]  , \label{koor-fo}%
\end{equation}
where the sum is over all reverse tableau $T$ of shape $\lambda$.
\end{definition}

The polynomial $P_{\lambda}(x;\tau,\alpha)$ is uniquely characterized by
certain vanishing conditions. To state these we set $\delta=\left(
n-1,\ldots,1,0\right)  $ and we define%
\[
\rho=\rho_{\tau,\alpha}=\left(  \rho_{1},\ldots,\rho_{n}\right)  ,\quad
\rho_{i}=\tau\delta_{i}+\alpha=\tau\left(  n-i\right)  +\alpha.
\]

\begin{theorem}
(\cite{Koor, Okounkov}) \label{Ok-poly} 
The polynomial $P_\lambda(x)=P_{\lambda}(x;\tau,\alpha)$ is in
$\mathcal{Q}$ and satisfy

\begin{enumerate}
\item $P_{\lambda}$ has degree $\leq2|\lambda|.$

\item The coefficient of $x_{1}^{2\lambda_{1}}\cdots x_{n}^{2\lambda_{n}}$ in
$P_{\lambda}$ is $1$.

\item $P_{\lambda}(\mu+\rho)=0\,$unless\thinspace$\lambda\subseteq\mu.$
\end{enumerate}
\end{theorem}

For future purposes we also define (\cite{Stanley}, \cite[(3.7)]{B-O})
\begin{equation}
k_{\mu}=\prod_{s\in\mu}
\left(
\tau l\left(  s\right)  
+a\left(  s\right)
+1\right)
\label{=k-mu}%
\end{equation}

\subsection{Uniqueness}

We prove a slight strengthening of Theorem \ref{Ok-poly}. For this we define
\[
\mathcal{Q}_{k}=\left\{  P\in\mathcal{Q}:\deg\left(  P\right)  \leq k\right\}
,\quad\Lambda^{d}=\left\{  \lambda\in\Lambda:\left\vert \lambda\right\vert
\leq d\right\}
\]

\begin{proposition}
\label{th:char}Any polynomial in $\mathcal{Q}_{2d}$ is characterized by its
values on the set
\[
\Lambda^{d}+\rho=\left\{  \lambda+\rho:\lambda\in\Lambda^{d}\right\}  \text{.}%
\]
\end{proposition}

\begin{proof}
Let $\mathcal{V}_{d}$ be the vector space of functions on the set $\Lambda
^{d}+\rho$, then we need to show that the restriction map
\begin{equation}
res:\mathcal{Q}_{2d}\rightarrow\mathcal{V}_{d} \label{=res}%
\end{equation}
is an isomorphism. Now $\mathcal{Q}_{2d}$ has an explicit basis given by the
set
\[
\left\{  \tilde{m}_{\lambda}:\lambda\in\Lambda^{d}\right\}  ,\quad\tilde
{m}_{\lambda}=\sum_{\sigma\in S_{n}}x_{\sigma\left(  1\right)  }^{2\lambda
_{1}}\cdots x_{\sigma\left(  n\right)  }^{2\lambda_{n}};
\]
thus both sides of (\ref{=res}) have the same dimension $\left\vert
\Lambda^{d}\right\vert $. Since $res$ is linear it suffices to prove that it
is surjective. For this we consider the \textquotedblleft$\delta
$-basis\textquotedblright\ of $\mathcal{V}_{d}$ given by%
\[
\delta_{\mu}\left(  \lambda+\rho\right)  =\delta_{\lambda\mu}\text{ for all
}\lambda,\mu\in\Lambda^{d}.
\]
Fix a total order on $\Lambda^{d}$ compatible with $\left\vert \lambda
\right\vert \geq\left\vert \mu\right\vert $. The restrictions
of Okounkov
polynomials
$\left\{
res\left(  P_{\mu}\right)  :\mu\in\Lambda^{d}\right\}  $ 
 belong to
$\mathcal{V}_{d}$,  and by Theorem \ref{Ok-poly} 
their expression in terms of the
$\delta$-basis is upper triangular with non-zero diagonal entries. Thus we can
invert this to write $\delta_{\mu}$ in terms of $res\left(  P_{\mu}\right)  $.
This proves the Proposition.
\end{proof}

\begin{theorem}
\label{char-oko} The Okounkov plynomial 
$P_{\lambda}(x;\tau,\alpha)$ is the unique polynomial in
$\mathcal{Q}$ satisfying

\begin{enumerate}
\item $P_{\lambda}$ has degree $\leq2|\lambda|.$

\item The coefficient of $x_{1}^{2\lambda_{1}}\cdots x_{n}^{2\lambda_{n}}$ in
$P_{\lambda}$ is $1$.

\item $P_{\lambda}(\mu+\rho)=0\,$ if \thinspace$\left\vert \mu\right\vert
\leq\left\vert \lambda\right\vert $ and $\mu\neq\lambda.$
\end{enumerate}
\end{theorem}

\begin{proof}
This follows immediately from Theorem \ref{Ok-poly} and Proposition
\ref{th:char}.
\end{proof}

\subsection{Explicit formulas for $\tau=1$}

In this section we give a determinantal formula for the Okounkov polynomials
when $\tau=1$. This involves the one variable polynomials discussed in the
next result.

\begin{lemma}
For $n=1$ and $l\in\mathbb{Z}_{+}$ the Okounkov polynomial is given by%
\begin{equation}
p_{l}\left(x;\alpha\right)=\prod_{i=0}^{l-1}\left(x^{2}-(i+\alpha)^{2}\right). \label{=pl}%
\end{equation}

\end{lemma}

\begin{proof}
We verify that $p_{l}(x;\alpha)$ satisfies the three conditions of Theorem
\ref{char-oko}. The first two are immediate, while for the third we need to
show
\begin{equation}
p_{l}\left(m+\alpha\right)=0\text{ for }m=0,1,\ldots,l-1, \label{=pl-vanish}%
\end{equation}
which follows from the formula $p_{l}(m+\alpha)=\prod_{i=0}^{l-1}%
(m+2\alpha+i)\prod_{i=0}^{l-1}(m-i)$.
\end{proof}

For $\lambda$ in $\Lambda$ we define an $n\times n$ matrix $A_{\lambda}$ and
its determinant $a_{\lambda}$ as follows%
\[
A_{\lambda}\left(  x;\alpha\right)  =\left(  p_{\lambda_j}\left(  x_{i}%
;\alpha\right)  \right)  _{1\leq i,j\leq n},\quad a_{\lambda}%
=\operatorname{det}A_{\lambda}.
\]
For $\delta=\left(  n-1,\ldots,1,0\right)  $ it is easy to see that
$a_{\delta}$ is the Vandermonde determinant $\prod_{i<j}\left(  x_{i}%
^{2}-x_{j}^{2}\right)  ,$ and is thus independent of $\alpha$.

\begin{theorem}
\label{thm-b} For $\tau=1$ the Okounkov polynomials are given by
\begin{equation}
P_{\lambda}\left(x;1,\alpha\right)=\frac{a_{\lambda+\delta}\left(  x;\alpha\right)
}{a_{\delta}\left(  x\right)  } \label{=P_lam1}%
\end{equation}

\end{theorem}

\begin{proof}
The proof is similar to \cite{Knop-Sahi}. Let us denote the right side of
(\ref{=P_lam1}) by $R_{\lambda}$. We will show that $R_{\lambda}$ satisfies
the conditions of Theorem \ref{char-oko}. The first condition is obvious.
Also, the top degree component of $R_{\lambda}$ is
\[
\frac{\operatorname{det}\left(  x_{i}^{2(\lambda_{j}+\delta_{j})}\right)
}{\operatorname{det}\left(  x_{i}^{2\delta_{j}}\right)  }=s_{\lambda}%
(x_{1}^{2},\ldots,x_{n}^{2})
\]
where $s_{\lambda}$ is the Schur polynomial; this implies the second
condition. To finish the proof it suffices to prove the third condition in the
form
\begin{equation}
\left\vert
 \mu\right
\vert \leq
\left\vert \lambda\right
\vert \text{ and
}R_{\lambda}\left(\mu+\rho\right)\neq0\,\implies\mu=\lambda\, \label{=m=l}%
\end{equation}

Suppose $\mu$ in $\Lambda$ satisfies the assumptions of (\ref{=m=l}). Since
$\mu+\rho$ has distinct components, the denominator in (\ref{=P_lam1}) is a
nonzero Vandermonde determinant, and so the numerator must be non zero.
Expanding the numerator we get
\[
\sum\nolimits_{\sigma\in S_{n}}(-1)^{\sigma}\prod\nolimits_{j}p_{\lambda
_{j}+\delta_{j}}\left(\mu_{\sigma(j)}+\delta_{\sigma(j)}+\alpha;\alpha\right)\neq0,
\]
and at least one term must be nonzero. Thus for some $\sigma\in S_{n}$ we must
have
\[
p_{\lambda_{j}+\delta_{j}}\left(\mu_{\sigma(j)}+\delta_{\sigma(j)}+\alpha
;\alpha\right)\neq0\text{ for all }j
\]
By (\ref{=pl-vanish}) we get
\begin{equation}
\mu_{\sigma(j)}+\delta_{\sigma(j)}\geq\lambda_{j}+\delta_{j}\text{ for all }j.
\label{=sig-j}%
\end{equation}
Summing this over $j$ we obtain
\begin{equation}
\left\vert \mu\right\vert +\left\vert \delta\right\vert \geq\left\vert
\lambda\right\vert +\left\vert \delta\right\vert \label{=l-m}%
\end{equation}
If the inequality in (\ref{=sig-j}) is strict for some $j$ then strict
inequality holds in (\ref{=l-m}), which contradicts the assumption that
$\left\vert \mu\right\vert \leq\left\vert \lambda\right\vert $. Thus equality
must hold in (\ref{=sig-j}) for all $j$, which implies%
\[
\sigma\left(  \mu+\delta\right)  =\lambda+\delta.
\]
Since $\lambda+\delta$ and $\mu+\delta$ are strictly decreasing sequences,
this forces $\sigma$ to be the identity permutation, and we get $\mu=\lambda$
as desired.
\end{proof}

\subsection{Explicit formulas for special partitions}

In this section we give explicit formulas for Okounkov polynomials
${P}_{\lambda}\left(  x;\tau,\alpha\right)  $ for certain special partitions
$\lambda$. For the reader's convenience we recall the definition of $\rho$
\[
\rho_{i}=\tau\delta_{i}+\alpha=\tau\left(  n-i\right)  +\alpha.
\]

\begin{theorem}
\label{egg} $P_{1^{j}}\left(  x;\tau,\alpha\right)  $ is the coefficient of
$t^{j}$ in the series expansion of%
\begin{equation}
\frac{\prod_{i=1}^{n}\left(  1+tx_{i}^{2}\right)  }{\prod_{i=j}^{n}\left(
1+t\rho_{i}^{2}\right)  } \label{=expansion}%
\end{equation}

\end{theorem}

\begin{proof}
Let $R_{j}$ denote the coefficient of $t^{j}$ in (\ref{=expansion}). We will
prove that $R_{j}$ satisfies the three conditions of Theorem \ref{char-oko}
for $\lambda=1^{j}$. The first two conditions are obvious. For the third
condition it suffices to prove that in the expression%
\begin{equation}
\frac{\prod_{i=1}^{n}\left(  1+t\left(  \mu_{i}+\rho_{i}\right)  ^{2}\right)
}{\prod_{i=j}^{n}\left(  1+t\rho_{i}^{2}\right)  } \label{=prod-exp}%
\end{equation}
the coefficient of $t^{j}$is $0$ if $\mu$ satisfies
\begin{equation}
\left\vert \mu\right\vert \leq j,\quad\mu\neq1^{j}. \label{=assume-mu}%
\end{equation}

However under assumption (\ref{=assume-mu}) we have $\mu_{j}=\mu_{j+1}%
=\cdots=\mu_{n}=0,$ and thus
\[
\mu_{i}+\rho_{i}=\rho_{i}\text{ for }i=j,\ldots,n\text{ }%
\]
It follows that in the expression (\ref{=prod-exp}) the denominator cancels
completely, leaving behind a polynomial in $t$ of degree $<j$. Hence the
coefficient of $t^{j}$ is $0$.
\end{proof}

\begin{corollary}
\label{egg-cor}
We have
\begin{equation}
P_{1^{j}}\left(x;\tau,\alpha\right)=\sum_{i_{1}<\cdots<i_{j}}\prod\nolimits_{k=1}%
^{j}\left(x_{i_{k}}^{2}-\rho_{i_{k}+j-k}^{2}\right). \label{eq:ks}%
\end{equation}
\end{corollary}

\begin{proof}
This follows by a direct computation from Theorem \ref{egg}. For an
alternative argument, see {\cite[Proposition 3.1]{Knop-Sahi}.}
\end{proof}

We next give an explicit formula for $P_{\lambda}\left(  x;\tau,\alpha\right)
$ for $\lambda=l^n:=l 1^{n}=\left(  l,l,\ldots,l\right)  .$

\begin{proposition}
We have%
\begin{equation}
P_{l^{n}}
\left(  x;\tau,\alpha\right)  =\prod\nolimits_{i=0}^{l-1}%
\prod\nolimits_{j=1}^{n}\left[  x_{j}^{2}-(i+\alpha^{2}\right]  \label{p-l1r}%
\end{equation}

\end{proposition}

\begin{proof}
It suffices to show that right side of (\ref{p-l1r}) satisfies the three
conditions of Theorem \ref{char-oko} for $\lambda=l^{n}.$ The first two
conditions are obvious. For the third it suffices to show that if%
\begin{equation}
\left\vert \mu\right\vert \leq nl,\quad\mu\neq l^{n}\label{=mu1}%
\end{equation}
then we have%
\begin{equation}
\prod\nolimits_{i=0}^{l-1}\prod\nolimits_{j=1}^{n}\left[  \left(  \mu_{j}%
+\rho_{j}\right)  ^{2}-(i+\alpha)^{2}\right]  =0.\label{=mu2}%
\end{equation}

But if $\mu$ satisfies (\ref{=mu1}) then we have $\mu_{n}<l$, which implies
\[
\mu_{n}+\rho_{n}=i+\rho_{n}=i+\alpha
\]
for some $i=0,1\ldots,l-1$, Thus one of the factors of (\ref{=mu2}) is $0.$
\end{proof}

\section{Properties for the eigenvalues of Shimura Operators}

\label{section-pro} We shall prove that the eigenvalues, i.e. the
Harish-Chandra homomorphism, of the Shimura operators are the Okounkov polynomials.

\subsection{Vanishing properties }

Before turning to our main results we prove an elementary lemma. Let $p$ be a
non-negative integer. Recall the Schmid's component $W_{\nu}$ of $S\left(
\mathfrak{\mathfrak{p}^{+}}\right)  $ and the $\mathfrak{g}$-representation
$V_{\lambda,p}$ in Lemma \ref{schli}. Recall the notation $W_{\nu,p}=W_{\nu
}\otimes\mathbb{C}_{\frac{p}{2}}$.

\begin{lemm+}
\label{lemma3.1} If $Hom_{K}\left(  W_{\nu,p},V_{\lambda,p}\right)  \neq0$
then $\nu\subseteq\lambda$.

\end{lemm+}

\begin{proof}
The Lie algebra $\mathfrak{k}+\mathfrak{p}^{-}$ is the parabolic subalgebra
opposite to $\mathfrak{k }+\mathfrak{p}^{+}$. Let $v_{\lambda, p}$ be a
non-zero highest weight vector in the representation space $(V_{\lambda, p},
\pi)$ of $\mathfrak{g}$. Then by the PBW theorem we have
\begin{equation}
\label{PBW}V_{\lambda, p}=\pi\left(U\left(  \mathfrak{p}^{-}\right)  U\left(
\mathfrak{k}\right)  \right) v_{\lambda, p} =\pi(U\left(  \mathfrak{p}^{-}\right)  )
\pi \left(U\left(  \mathfrak{k}\right)  \right) v_{\lambda, p}.
\end{equation}
The space $
\pi (U\left(  \mathfrak{k}\right)  ) v_{\lambda, p}=
\pi\left(  \mathfrak{k}\right)  v_{\lambda, p}=W_{\lambda, p} $ is
a highest weight representation of $\mathfrak{k}$ with highest weight
$\sum_{j=1}^{n}(\lambda_{j}+\frac p2)\gamma_{j} $ when restricted to
$\mathfrak{t}_{-}$. If $Hom_{K}\left(  W_{\nu, p}, V_{\lambda, p}\right)
\neq0$, equivalently $W_{\nu, p} $ occurs in $\pi\left(U\left(  \mathfrak{p}%
^{-}\right)  \right)W_{\lambda, p} $ then $\sum_{j=1}^{n}\left(\nu_{j}+\frac
p2\right)\gamma_{j}$ must be of the form $\sum_{j} \left(\nu_{j}+\frac p2\right) \gamma_{j} =
\sum_{j}\left(\lambda_{j}+\frac p2+\mu_{j}\right) \gamma_{j} $ where $\sum_{j} \mu
_{j}\gamma_{j} $ is a weight of $U\left(  \mathfrak{p}^{-}\right)  $, by  \cite[Theorem 20.2]{Hum}. But then any such $\mu$ is of the form
$-\sum_{i}\mu_{i}\gamma_{i}$ for some $\mu_{i}\geq0$, proving our claim.
\end{proof}



\begin{theorem}
\label{thm-1} Let 
$\eta_p(\mathcal L_{\mu})$
 be the Harish-Chandra homomorphism of 
$\mathcal L_{\mu}$ defined in  (\ref{eq:eta-p}),
then
\[
\eta_{p}\left(  \mathcal{L}_{\mu}\right)  \left(  \lambda+\frac{p}{2}%
+\rho\right)  =0\text{ unless }\mu\subseteq\lambda.
\]

\end{theorem}

\begin{proof}
 We equip $V_{\lambda,p}$ with a
$U$-invariant unitary inner product. Now $\mathcal{L}_{\mu}$ is a sum of
elements of the form $\bar{\xi}\xi$, where $\{\xi\}$ is a basis of $W_{\mu
}\subset S(\mathfrak{p}^{+})$. By Schur lemma the invariant differential
operator $\mathcal{L}_{\mu}$ acts by the scalar $c=\eta_{p}\left(
\mathcal{L}_{\mu}\right)  (\lambda+\frac{p}{2}+\rho)$ on $V_{\lambda,p}$. Let
$v_{p}$ be the $(K,\iota^{p})$-spherical vector in $V_{\lambda,p}$, normalized
to have $\left(  v_{p},v_{p}\right)  =1$. The $(K,\iota^{-p})$-spherical
function in $V_{\lambda,p}\subset C^{\infty}(G/K,\iota^{p})$ is of the form
\[
\Phi_{\lambda,p}(g)=(\pi_{\lambda}(g^{-1})v_{p},v_{p})
\]
Performing differentiation by $\mathcal{L}_{\mu}$ and evaluating at $g=1\in G$
we get
\[
c=\left(  \pi_{\lambda}(\mathcal{L}_{\mu})v,v\right) 
 =\sum\left(\pi_{\lambda}
(\xi)v_{p},\pi_{\lambda}(\bar{\xi})^{\ast}v_{p}\right)=
\sum\left(\pi_{\lambda}(\xi
)v,\pi_{\lambda}(\xi)v\right).
\]
Here we have used the fact that $\pi_{\lambda}(x)^{\ast}=\pi_{\lambda}(\bar
{x})$ since $\pi_{\lambda}$ is a unitary representation of Lie algebra
$\mathfrak{k}+i\mathfrak{p}$ of $U$. But the vectors $\pi_{\lambda}(\xi)v_{p}%
$, $\xi\in W_{\mu}$, are in the $K$-subspace of $V_{\lambda,p}$ of highest
weight $\mu+\frac{p}{2}$, which is vanishing by Lemma \ref{lemma3.1}.
\end{proof}

\begin{remark}
By the same argument above there exists polynomial 
$\eta_p(\mathcal M_{\mu}) $ such that
$\mathcal{M}_{\mu}$ acts on $V_{\lambda, p}$ by the scalar $
\eta_p(\mathcal M_{\mu}) 
\left(
\lambda+\frac p2+\rho\right)  $. Moreover the polynomial $\eta_p(\mathcal M_{\mu}) $ is
related to $\eta_p(\mathcal L_{\mu })
$ by
\[
\eta_{-p}(\mathcal M_{\mu}) =
\eta_p(\mathcal L_{\mu })
\]
for $\lambda\in\mathbb{C}^{n}$. This relation is a simple consequence of the
following observation: If $f\in C^{\infty}(G/K, p)$ then $\bar f\in C^{\infty
}(G/K, -p)$ and
\[
X\bar f =\overline{\bar X f}
\]
where $X\to\bar X$ is the complex conjugation relative to the real form
$\mathfrak{g}_{0}$. The $U$-representations appearing in $C^{\infty}(G/K, -p)$
are the same as in $C^{\infty}(G/K, -p)$ by Lemma 2.2 and are of the form
$\lambda+\frac p2$. Thus $\eta_{-p}
(\mathcal M_\mu)
(\lambda+ \frac p2+\rho) =
\eta_{-p}(\mathcal L_{\mu})(\lambda+ \frac p2+\rho)$ for $\mu\in\Lambda$, but $\Lambda$ is Zariski
dense in $\mathbb{C}^{n}=\mathfrak{a}^{*}$ so it holds also on $\mathbb{C}%
^{n}$. See further \cite{gz-shimura}.
\end{remark}


\begin{rema+}
Let $p=0$. The Harish-Chandra spherical function $\phi_{x}$ on $G/K$ in
\cite[Ch. IV, Theorem 4.3]{He3} and \cite{Shimeno-jfa} is our $\Phi_{ix,
0}$ Thus there is a change of variable $x\to ix$ from the
parameterization in \cite{He3} to ours here.
\end{rema+}

\subsection{Eigenvalue polynomials 
$\eta_p(\mathcal L_{\mu})$
in terms of $P_{\mu}$}


We can now find the precise relation between $\eta_p(\mathcal L_{\mu})$  and Okounkov's
BC-interpolation polynomials. So let $P_{\lambda}\left(x,\tau,\alpha\right)$ be as above
the BC-type interpolation polynomials with two parameters $(\tau,\alpha)$ with
the normalization that the coefficient of $m_{\lambda}$ is $1$. We prove now
one of our main results, stated as Theorem 1.1 in Section 1. Recall the
reproducing kernel $K_{\lambda}(z,w)$ of the space $W_{\lambda}$ in Section
\ref{hua-schmid} equipped with the Fock norm.

We define%
\begin{equation}
\tau=\tau(d):=\frac{d}{2},\,\alpha:=\alpha(b,p)=\frac{b+1+p}{2}%
,\label{sz-vs-ko}%
\end{equation}
so that $\rho+\frac{p}{2}=\rho(\tau,\alpha)$. Also recall the constant
$k_{\mu}$ defined in (\ref{=k-mu})

\begin{theorem}
\label{thm1} 
The Harish-Chandra image of $\mathcal{L}_{\mu}$ is
\[
\eta_{p}\left(  \mathcal{L}_{\mu}\right)  =k_{\mu}P_{\mu}\left(x;\tau,\alpha\right)
\]
where $(\tau,\alpha)$ are as in (\ref{sz-vs-ko}),
\end{theorem}

\begin{proof}
It follows from Theorem \ref{thm-1} and Corollary \ref{char-oko} that
$\eta_p(\mathcal L_{\mu})(x)$ is a scalar multiple of $P_{\mu}^{{}}\left(x; \tau,\alpha\right)$,
$\eta_p(\mathcal L_{\mu})(x)=
k P_{\mu}^{{}}\left(x;\tau,\alpha\right)$. To find the scalar constant
$k$ we compare their leading terms. Recall the Cartan subspace
$\mathfrak{a}=\sum_{j}\mathbb{C}e_{j}$ of $\mathfrak{p}$. Now each element
in $\mathfrak{p}$ can be written as $u=u^{+}+u^{-}$, and $u^{\pm}=\frac{1}%
{2}(u\pm i\tilde{u})$ for $\tilde{u}=[Z_{0},u]\in\mathfrak{p}$ with
$Z_{0}$ defining the complex structure on $\mathfrak{p}_{0}$. In particular we
have $e_{j}^{+}=\frac{1}{2}\left(e_{j}+i\tilde{e}_{j}\right)$. Any $x=\sum
_{j}x_{j}(2\epsilon_{j})\in\mathfrak{a}^{\ast}$ can be extended to an
element in $\mathfrak{p}^{\ast}$, and thus $x(e_{j})=2x_{j}$,
$x(e_{j}^{+})=x_{j}$. It follows then from the definition of
$\mathcal{L}_{\mu}$ that the Harish-Chandra homomorphism 
$\eta_p(\mathcal L_{\mu})$ 
of
$\mathcal{L}_{\mu}$ has its leading term the polynomial
\[
x=\sum_{j}x_{j}e_{j}\in\mathfrak{a}_{0}^{\ast}\mapsto K_{\mu}
\left(\sum_{j}%
x_{j}e_{j}^{+},\sum_{j}x_{j}e_{j}^{+}\right).
\]
Now the sum of $\sum_{|\mu|=m}K_{\mu}\left(\sum_{j}x_{j}e_{j}^{+},\sum_{j}%
x_{j}e_{j}^{+}\right)$ is
\[
\sum_{|\mu|=m}K_{\mu}\left(\sum_{j}x_{j}e_{j}^{+},\sum_{j}x_{j}e_{j}^{+}\right)=\frac
{1}{m!} \left(x_{1}^{2}+\cdots+x_{n}^{2}\right)^{m}%
\]
by the definition of the reproducing kernel $K_{\mu}$. On the other hand the
top homogeneous term of $P_{\mu}\left(x;\tau,\alpha\right)$ is precisely the monic Jack
polynomial $P_{\mu}^{Jac}\left(x_{1}^{2},\cdots,x_{n}^{2}\right)$ with parameter $\tau$,
\cite{Koor}, thus the constant $k$ is precisely the coefficient $k_{\mu
}^{\prime}$ in the expansion
\[
\frac{1}{m!}
\left(x_{1}^{2}+\cdots+x_{n}^{2}\right)^{m}=\sum_{|\mu|=m}k_{\mu}^{\prime
}P_{\mu}^{Jac}\left(x_{1}^{2},x_{2}^{2},\cdots,x_{n}^{2}\right).
\]
But it is well-known (\cite{Stanley}, \cite[(iii)-(iv)-(vii), p.~1319]{Yan})
that the constant $k_{\mu}^{\prime}$ is given by (\ref{=k-mu}).
\end{proof}

\begin{rema+}
We can also give a different proof of the evaluation formula for the constant
$k_{\mu}$ above. Let $e^{+}=e_{1}^{+}+\cdots+e_{n}^{+}$ be the sum of the
strongly orthogonal positive root vectors. Recall \cite[Lemma 3.1]{FK} that
\begin{equation}
\label{eq:K-phi}
K_{\mu}\left(z, {e^{+}}\right)= \frac{\beta_{\mu}} 
{
\left(\frac d2
(n-1) +1\right)_{\mu}
} 
\psi_{\mu}(z),
\end{equation}
where $\psi_{\mu}$ is the spherical polynomial of the $K$-homogeneous space
$Ke^{+}$ (i.e., the Shilov boundary of $G/K$) normalized by $\psi_{\mu}(e^{+})=1$
\begin{equation}
\label{pi-lambda}\beta_{\mu}:= \prod_{1\le i<j\le n} \frac{\mu
_{i}-\mu_{j}+ \frac d2 (j-i)} {\frac d2(j-i)} \, \frac{(\frac d2
({j-i+1}))_{\mu_{i}-\mu_{j}}} {(\frac d2({j-i-1})+1)_{\mu
_{i}-\mu_{j}}},
\end{equation}
and
\begin{equation}
\label{gen-Pochammer}
(a)_{\mu}:= \prod_{i=1}^n\left(a-\frac d2 (i-1) \right)_{\mu_i}
= \prod_{i=1}^n\frac{\Gamma
\left(
a-\frac d2 (i-1)+{\mu_i} 
\right)}
{\Gamma\left(a-\frac d2(i-1)
\right)}
\end{equation}
is the generalized Pochammer symbol.
See e.g. \cite[(2.6)-(2.7)]{Englis-Zhang} where our $\beta_{\mu}$ is
denoted by $\pi_{\mu}$. (The coefficient $\frac{\beta_{\mu}} 
{\left(\frac
d2 (n-1) +1 \right)_{\mu}}$ is now independent of the root multiplicity $2b$.) Now the top homogeneous
term of $
\eta_p(\mathcal L_{\mu})$ is given by $
K_{\mu}\left(\sum_{j}x_{j}e_{j}^{+},\sum_{j}%
x_{j}e_{j}^{+} \right)$, which in turn is (\cite{FK})
\[
K_{\mu}\left(\sum_{j}x_{j}e_{j}^{+},\sum_{j}x_{j}e_{j}^{+}\right)=K_{\mu}\left(\sum_{j}%
x_{j}^{2}e_{j}^{+},e^{+}\right)=\frac{\beta_{\mu}}
{
\left(
\frac{d}{2}(n-1)+1
\right)_{\mu}}%
\psi_{\mu}
\left(\sum_{j}x_{j}^{2}e_{j}^{+}\right),
\]
where the last equation is just (\ref{eq:K-phi}). Now $\psi_{\mu}
\left(\sum
_{j}x_{j}^{2}e_{j}^{+}
\right)=
\psi_{\mu}\left(x_{1}^{2},\cdots,x_{n}^{2}\right)$ is the Jack
symmetric polynomial $\psi_{\mu}\left(x_{1}^{2},\cdots,x_{n}^{2}
\right)=\frac{1}{P_{\mu
}^{Jac}(1^{n})}
P_{\mu}^{Jac}\left(x_{1}^{2},\cdots,x_{n}^{2}\right)$, whereas the
Okounkov polynomial $P_{\mu}^{{}}$ has the same leading term as $P_{\mu}%
^{Jac}\left(x_{1}^{2},\cdots,x_{n}^{2}\right)$; see \cite{Koor}. Thus the constant
$k=k_{\mu}$ is $k=\frac{\beta_{\mu}}
{
\left(
\frac{d}{2}(n-1)+1
\right)_{\mu}}\frac
{1}{P_{\mu}^{Jac}(1^{n})}$. Now by the known evaluation formula (see e.g.
\cite[(4.8)]{Koor})
\begin{equation}
{P_{\mu}^{Jac}(1^{n})}=\prod_{1\leq i<j\leq n}
\frac{
\left((j-i+1)\frac{d}{2}
\right)_{\mu_{i}-\mu_{j}}}
{
\left(
(j-i)\frac{d}{2}
\right)_{\mu_{i}-\mu_{j}}}, \label{p-jac-at-1}%
\end{equation}
we can write $k$
\[
k=\frac{1}
{
\left(
\frac{d}{2}(n-1)+1
\right)_{\mu}}
\prod_{1\leq i<j\leq n}\frac{\mu_{i}%
-\mu_{j}+\frac{d}{2}(j-i)}
{\frac{d}{2}(j-i)}\,
\frac{\left(\frac{d}{2}({j-i}%
)
\right)_{\mu_{i}-\mu_{j}}}
{\left(
\frac{d}{2}({j-i-1})+1
\right)_{\mu_{i}-\mu_{j}}}.
\]
This can be simplified using the Gamma function
\[
k=\frac{1}{
\left(
\frac{d}{2}(n-1)+1
\right)_{\mu}}
\prod_{1\leq i<j\leq n} \frac
{\gamma(\mu_{i}-\mu_{j}, j-i)
}
{\gamma(0, j-i)}
\]
 where
$$
\gamma(x, j-i):=
\frac{\Gamma
\left(x+1+\frac{d}{2}(j-i)
\right)
}
{
\Gamma
\left(
x+1+\frac{d}%
{2}(j-i-1)
\right)}.
$$
By a straightforward computation using (\ref{gen-Pochammer}) we find
\[%
k =\prod_{i=1}^{n}\frac{1}{
\Gamma
\left(
1+\frac{d}{2}(n-i)+\mu_{i}
\right)}
\prod_{1\leq
i<j\leq n} \gamma(\mu_{i}-\mu_{j}, j-i),
\]
which  is precisely (\ref{=k-mu}),
by
\cite[Proposition 3.5]{B-O}.
\end{rema+}




We can now describe the Shimura sets in terms
of the Okounkov polynomials. Using
Theorem \ref{thm1}  and the definition of 
the sets $\mathcal{A}$ and $\mathcal{G}$ we have
\begin{corollary}
\label{cor1-2}
The Shimura sets are given explicitly as follows:%
\begin{align*}
\mathcal{A}
  &  =\left\{  x:q_{\lambda}\left(  x\right)  \geq0\text{ for all
}\lambda\right\}  ,\\
\mathcal{G}  &  =\left\{  x:q_{1^{j}}\left(  x\right)  \geq0\text{ for all
}j\right\}.
\end{align*}
\end{corollary}

The following is a restatement of Theorem 1.2, the notation being the same,
and it is immediate consequence of Corollaries
\ref{cor1-2} and  \ref{egg-cor}.

\begin{theorem}
\label{G-set} The Shimura set $\mathcal{G}$
is also given by $\mathcal{G}=\left\{  \xi:\varphi_{j}\left(  \xi\right)
\geq0\text{ for all }j\right\}  $.
\end{theorem}

\section{Further analysis of the Shimura sets
}
In the rest of the paper we let $p=0$ and write $\eta_p(\mathcal L_{\mu})
=\eta(\mathcal L_{\mu})$, namely
we consider the trivial line bundle over $G/K$.

\subsection{ The Shimura sets
and  unitary spherical representations
of $G$}

We introduce now the set 
\[
\mathcal{U}=\{x\in\mathbb{C}^{n};\text{the spherical function $\Phi_{x}$ is
positive definite}\}.
\]
In other words $\mathcal{U}$ is the set of unitary spherical representations.
This set has been studied intensively, and in \cite{Kn-Sp} it is determined
for the group $G=U(N,2)$.

\begin{prop+}
\label{uni-pos} We have
\[
\mathcal{U}\subseteq\mathcal{A} \subseteq\mathcal{G}.
\]
\end{prop+}

\begin{proof}
Let $x\in\mathcal{U}$. The spherical function $\Phi_{x}$ defines a unitary
irreducible representation $(H, \pi)$ of $G$ with a $K$-fixed vector $v$ so
that $\Phi_{x}$ is the matrix coefficient
\[
\Phi_{x}(g)=\left(v, \pi(g)v\right),
\]
where $( , )$ is the Hilbert Hermitian product in $H$; see e.g.~\cite[Ch.~IV]{He3}. For any element $X\in\mathfrak{p}$ we have
\[
X \Phi_{x}(g) = \left(v, \pi(X )v\right)= 
\left(\pi(X)^{*}v, v\right) =-\left(\pi(\bar X)v, v\right)
\]
where $\bar X$ is the complex conjugation with respect to the real form
$\mathfrak{g}_{0}$ in $\mathfrak{g}$. Now let $\mathcal{L}_{\mu}$ act on
$\Phi_{x}$ and evaluate at $g=e$. We have
\[%
\begin{split}
(-1)^{|\mu|} \mathcal{L}_{\mu}\Phi_{x}(e)  &  =
(-1)^{|\mu|}\sum_{\alpha}
\left
(v,
\pi(\xi_{\alpha}^{*}) \pi(\xi_{\alpha})v
\right
)\\
&  =(-1)^{|\mu|} \sum_{\alpha}
\left
( \pi(\xi_{\alpha}^{*})^{*}v, \pi(\xi_{\alpha
})v
\right
)
\\
&  =(-1)^{|\mu|}\sum_{\alpha}
\left
( \pi(\bar{\xi_{\alpha}^{*}})v, \pi(\xi_{\alpha
})v
\right
)\\
&  =\sum_{\alpha}
\left(
 \pi(\xi_{\alpha})v, \pi(\xi_{\alpha})v\right) \ge0,
\end{split}
\]
proving $\mathcal{U}\subseteq\mathcal{A}$.
\end{proof}

\subsection{Positivity for real parameters}

We shall study a real version of the sets $\mathcal{A}, \mathcal{G},
\mathcal{U}$. Denote $\mathcal{A}_{0}=\mathcal{A}\cap C $, $\mathcal{G}_{0}=
\mathcal{G}\cap C$, $\mathcal{U}_{0}=\mathcal{U}\cap\mathcal{C }$ 
where
$\mathcal{C }\text{ }=\left\{x: x_{1}\ge\cdots\ge x_{n}\geq0\right\}  $
is a Weyl chamber.

\begin{theo+}
\label{square} Suppose the rank $n>1$ then we have
\[
[0, \rho_{n}]^{n} \cap\mathcal{C }
\subsetneq\mathcal{A}_{0} \cap\mathcal{C }
\]

\end{theo+}

\begin{proof}
We shall need an explicit formula for $P_{\lambda}^{{}}(x)$ by Koornwinder
\cite{Koor}.

By Theorem \ref{thm1} we have, using
the notation
$q_{\lambda}$ in (\ref{=q}), that
\begin{equation*}
q_{\lambda}(x)=(-1)^{|\lambda|}P_{\lambda}^{{}}(x)=k_{\mu}\sum_{T}\psi
_{T}\prod_{s\in\lambda}
\left
(
(a_{\lambda}^{\prime}(s)+\frac{d}{2}(n-T(s)-l_{\lambda
}^{\prime}(s))+\frac{b+1}{2})^{2}-x_{T(s)}^{2}
\right)   
\end{equation*}
with $k_{\mu}$ being positive. Now if $x\in\lbrack0,\rho_{n}]^{n}$, i.e, if
$0\leq x_{j}\leq\rho_{n}\forall j$, we have for any fixed $T$ in the sum and
$s=(i,j)\in\lambda$ in the product, writing $T(s)=k$, that $a_{\lambda
}^{\prime}(s)=j-1$, $l_{\lambda}^{\prime}(s)=i-1$ and
\[%
\begin{split}
a_{\lambda}^{\prime}(s)+\frac{d}{2}
\left(
n-k-l_{\lambda}^{\prime}(s)
\right)
+\frac
{b+1}{2}  &  \geq\frac{d}{2}
\left(
n-k-l_{\lambda}^{\prime}(s)
\right)+
\frac{b+1}{2}\\
&  =
\frac{d}{2}
\left(n-k-i+1
\right)+\frac{db+1}{2}\\
&  =\rho_{k+i-1}\geq\rho_{n}\geq x_{k}.
\end{split}
\]
Here we have used the fact that $T(i,j)$ is strongly decreasing in $i$,
implying $T(s)=T(i,j)=k\leq n-i+1$ and $\rho_{k+i-1}$ makes sense. Thus each
factor in the product is nonnegative and $
(-1)^{|\lambda|}\eta(\mathcal L_{\lambda})(x)
\geq0,$ proving
$x\in\mathcal{A_0}$. The element $\rho$ is in $
\mathcal
A_0$ since it is a zero point of all 
$\eta(\mathcal L_{\lambda})$, but $\rho\notin [0, \rho_n]^n$.
This finishes the proof.
\end{proof}


Note that if the rank $n=1$ then the three sets are the same
\[
\mathcal{A}= \mathcal{G}= \mathcal{U}=[-\rho, \rho]\cup i\mathbb{R},
\]
and
\[
\mathcal{A}_{0}= \mathcal{G}_{0}= \mathcal{U}_{0}=[0, \rho]
\]
In other words, the set of unitary spherical representations are characterized
by one relation, namely $x^{2}-\rho^{2} \le0$, with the complementary series
parameters corresponding to the real points.

\subsection{The case of rank two domains $(\mathfrak{g}_{0},\mathfrak{k}%
_{0})=(\mathfrak{u}(b+2,2),\mathfrak{u}(b+2)+\mathfrak{u}(2))
$ and
$(\mathfrak{sp}(2,\mathbb{R}),\mathfrak{u}(2))$
}

We shall determine the set $\mathcal{A}_{0}$
 for the domains $G/K$ of rank
$n=2$ and with $d=2$, namely $(\mathfrak{g}_{0},\mathfrak{k}_{0})$ being the
pair $(\mathfrak{u}(b+2,2),\mathfrak{u}(b+2)+\mathfrak{u}(2))$ and prove 
an inclusion for the pair 
$(\mathfrak{sp}(2,\mathbb{R}),\mathfrak{u}(2))$;
we refer the two pairs as $I_{2,2+b}$ and  $II_{2}$. The
variable $x$ will be in the Weyl Chamber $\mathcal C\subset
\mathbb{R}_{\geq0}^{2}$ throughout the discussions
below. Recall the
Pochammer symbol 
$(a)_{m}=(a)(a+1)\cdots(a+m-1)$ 
introduce its multiparameter version%
\[
(a_{1},\ldots,a_{p})_{k}=(a_{1})_{k}\cdots,(a_{p})_{k},
\]
To simplify notation still further we will write $\left(  a\pm x\right)  _{k}$
for $\left(  a+x\right)  _{k}\left(  a-x\right)  _{k}$.

In \cite[(10.13)]{Koor} Koornwinder found explicit formulas for the
interpolation polynomials $P_{(m_{1},m_{2})}^{{}}(x_{1},x_{2})$ of rank two in
terms of hypergeometric series $\,\mbox{}_{p}F_{q}\!\left(
\genfrac{}{}{0pt}{}{a}{b}%
;t\right)  $, $a=(a_{1},\cdots,a_{p})$, $b=(b_{1},\cdots,b_{p})$. We shall be
only dealing with the series evaluated at $t=1$. To ease notation we write
\[
\,F\!\left(
\genfrac{}{}{0pt}{}{a_{1},\cdots,a_{p}}{b_{1},\cdots,b_{q}}%
\right)  =\,\mbox{}_{p}F_{q}\!\left(
\genfrac{}{}{0pt}{}{a_{1},\cdots,a_{p}}{b_{1},\cdots,b_{q}}%
;1\right)  =\sum_{k=0}^{\infty}\frac{(a_{1},\cdots,a_{p})_{k}}{(b_{1}%
,\cdots,b_{q})_{k}}\frac{1}{k!}%
\]
and its partial sum
\[
\,F^{[m]}\!\left(
\genfrac{}{}{0pt}{}{a_{1},\cdots,a_{p}}{b_{1},\cdots,b_{q}}%
\right)  =\sum_{k=0}^{m}\frac{(a_{1},\cdots,a_{p})_{k}}{(b_{1},\cdots
,b_{q})_{k}}\frac{1}{k!}%
.\]
In the formulas below we adapt also the short-hand notation 
$\alpha \pm \beta$ to indicate that  the both terms appear 
in parallell positions.

\begin{lemm+}
\label{p-f-r=2} (\cite{Koor}) The Okounkov polynomial $q_{(m_{1},m_{2})}(x)= (-1)^{m_{1}+m_{2}}P_{(m_{1},m_{2})}^{{}}(x)
$ of two variables $x=(x_{1},x_{2})$ with the parameter $(\tau,\alpha)$ is
given in terms of ${}_{4}F_{3}$-series by
\begin{multline}
q_{(m_{1},m_{2})}(x)=
(\rho_{2}\pm x_{1},\rho_{2}\pm
x_{2})_{m_{2}}\,(m_{2}+\rho_{1}\pm x_{1})_{m_{1}-m_{2}}\\
\times\,F\!\left(
\genfrac{}{}{0pt}{}{-m_{1}+m_{2},
m_{2}+\rho_{2}\pm x_{2}, \frac{d}{2}}
{1-m_{1}+m_{2}-\frac{d}{2},
m_{2}+\rho_{1}\pm x_{1}}%
\right)  .
\end{multline}
In particular if $d=2$ the polynomial $q_{(m_{1},m_{2})}%
^{{}}(x_{1},x_{2})$ can be written in terms of the partial sum of an ${}%
_{3}F_{2}$-series
\begin{multline}
q_{(m_{1},m_{2})}^{{}}(x_{1},x_{2})
=(\rho_{2}\pm x_{1}%
,\rho_{2}\pm x_{2})_{m_{2}}
\,(m_{2}+\rho_{1}\pm x_{1})_{m_{1}-m_{2}}\\
\times\,F^{[m_{1}-m_{2}]}\!\left(
\genfrac{}{}{0pt}{}{m_{2}+\rho_{2}\pm x_{2},\frac{d}{2}}
{m_{2}+\rho_{1}\pm x_{1}
}\right).
\end{multline}

\end{lemm+}

Denote
\begin{equation}
\label{eq:S-test}R(x_1, x_2):= \,F\!\left(
\genfrac{}{}{0pt}{}{\rho_{2}\pm x_{2}, \frac d2}{\rho_{1}\pm x_{1}}%
\right)
\end{equation}

\begin{theo+}
\label{thm4.4} Let
\[
\mathcal{B}=\{x\in\mathcal{C}\mid q_{1, 0}(x)\ge0, \quad q_{1, 1}(x)\ge0,
R(x)\ge0\}
\]
Then the set $\mathcal{A}_{0}$ of real points $\lambda$ for the positivity of
all $q_{\mu}(\lambda)$ is 
$\mathcal{A}_{0}=\mathcal{B}$ if
$(\mathfrak{g}_{0}, \mathfrak{k}_{0})$ if of type $I_{2, 2+b}$, and
$\mathcal{A}_{0}\subseteq\mathcal{B}$ for type $II_{2}$.
\end{theo+}

\begin{proof}
To ease notation we take all $x$ below to be in the first quarter $x_{1},
x_{2}\ge0$ instead of the Weyl chamber $\mathcal{C}$. It follows immediately
from the formulas in Lemma \ref{p-f-r=2} that $q_{(1, 0)}(x)\ge0, q_{(1,
1)}(x)\ge0$ if and only if $x\in[0, \rho_{2}]^{2}$ or $x_{1}, x_{2}\ge\rho
_{2}, \Vert x\Vert\le\Vert\rho\Vert$, namely, $x$ is in the square $[0,
\rho_{2}]^{2}$ or in disc $\{\Vert x\Vert\le\Vert\rho\Vert\}$ cut by the
square $[\rho_{2}, \rho_{1}]^{2}$, i.e. $\{\Vert x\Vert\le\Vert\rho\Vert\}\cap
[\rho_{2}, \rho_{1}]^{2}$. However the triangle $[0, \rho_{2}]^{2}%
\cap\mathcal{C}$ is in $\mathcal{A}_{0}$ by Theorem \ref{square} above so we
need only consider $x$ in the square $[\rho_{2}, \rho_{1}]^{2}$ and we
restrict $x$ to this square.

We prove first the inclusion $\mathcal{A}_{0}\subseteq\mathcal{B}$ for
$d=1,2$. Note first that $\rho_{1}=\rho_{2}+\frac{d}{2}$ and observe that
$\rho_{2}-x_{2}\leq0$ and $\rho_{2}-x_{2}+l\geq0$ if $l\geq1$ for all $x$ in
the square $[\rho_{2},\rho_{1}]^{2}$. Suppose $q
_{(m,0)}^{{}}(x_{1},x_{2})=(-1)^{m}P_{(m,0)}^{{}}(x_{1},x_{2})\geq0$ 
for all $m$. We fix $N>0$ and let $m\geq N$. Denote the
partial sum in $a_{(m,0)}^{{}}(x_{1},x_{2})$ by
\[
f_{m,N}(x):=\sum_{j=0}^{N}\frac{(m)_{j}^{-}(\rho_{2}\pm x_{2},%
)_{j}(\frac{d}{2})_{j}}
{
\left(
m+\frac{d}{2}-1\right)_{j}^{-}
\left(\rho_{1}\pm x_{1} \right)_{j}j!}.\]
Now by the above observation $q_{(m,0)}^{{}}(x_{1},x_{2})$ has leading
term $1$ with the rest being nonpositive, we have
\[
f_{m,N}(x)\geq q_{(m,0)}^{{}}(x_{1},x_{2})\geq0
\]
Letting $m\rightarrow\infty$ we find
\[
\sum_{j=0}^{N}\frac{\left(\rho_{2}\pm x_{2}\right)_{j}
(\frac{d}{2})_{j}%
}{\left(\rho_{1}\pm x_{1}\right)_{j}j!}=\lim_{m\rightarrow\infty}%
f_{m,N}(x)\geq0.
\]
Now take the limit $N\rightarrow\infty$:
\[
R(x)=\lim_{N\rightarrow\infty}\sum_{j=0}^{N}\frac{
\left(\rho
_{2}\pm x_{2}\right)_{j}(\frac{d}{2})_{j}}
{\left(\rho_{1}\pm x_{1}\right)_{j}j!}%
\geq0,
\]
proving $\mathcal{A}_{0}\subseteq\mathcal{B}$.

Suppose now $d=2$,  $x\in\mathcal{B}$ and is in the square
$[\rho_{2}, \rho_{1}]^{2}$. Thus $R(x)\geq
0$. If $m_{1}=m_{2}\geq1$ then $q_{(m_{1},m_{2})}^{{}}$ is a product
of $m_{1}$ pairs of nonpositive numbers and is nonnegative. Let $m_{1}=m\geq
m_{2}=0$.  By  Lemma \ref{p-f-r=2} 
the polynomial $q_{(m,0)}^{{}}(x_{1},x_{2})$ is a partial sum of
an ${}_{3}F_{2}$ series,
is
\[%
q_{(m,0)}^{{}}(x_{1},x_{2})  =\left(\rho_{1}\pm x_{1}\right)_{m}\,F^{[m]}\!\left(
\genfrac{}{}{0pt}{}{\rho_{2}\pm x_{2},1}{\rho
_{1}\pm x_{1}}%
\right)
\]
with the factor $(\rho_{1}\pm x_{1})_{m}\geq0$. The second factor
is
\[
\,F^{[m]}\!\left(
\genfrac{}{}{0pt}{}{-m,\rho_{2}\pm x_{2},\frac{d}{2}}{1-m-\frac
{d}{2},\rho_{1}\pm x_{1}}%
\right)  =\sum_{j=0}^{m}\frac{
\left(\rho_{2}\pm x_{2}\right)_{j}(1)_{j}%
}{
\left(
\rho_{1}\pm x_{1}
\right)_{j}j!}.
\]
All terms in the sum are nonpositive except the leading term $1$. Thus adding
infinitely many negative terms we find
\[
\begin{split}
\quad\,F^{[m]}\!
\left(
\genfrac{}{}{0pt}{}{\rho_{2}\pm x_{2}, 1}{\rho_{1}\pm x_{1},}%
\right)   &  =\sum_{j=0}^{m}\frac{
\left(\rho_{2}\pm x_{2}\right)_{j}(1)_{j}%
}{(\rho_{1}\pm x_{1})_{j}j!}\\
&  \geq\sum_{j=0}^{\infty}\frac{\left(\rho_{2}\pm x_{2}\right)_{j}(1)_{j}%
}{\left(\rho_{1}\pm x_{1}\right)_{j}j!}\\
&  =\,F\!\left(
\genfrac{}{}{0pt}{}{\rho_{2}\pm x_{2},1}{\rho_{1}\pm x_{1}}%
\right) \\
&  =R(x)\geq0.
\end{split}
\]

Now if $m_{1}>m_{2}>0$ the positivity of $q_{(m_1, m_2)}(x)$ for
$x\in\lbrack\rho_{2},\rho_{1}]^{2}$, $\rho_{2}\leq x_{2}\leq\rho_{2}%
+1=\rho_{1}$ follows immediately using Lemma \ref{p-f-r=2} as all terms in the
summation of $F^{[m_{1}-m_{2}]}$ are positive.

The proof is now completed.
\end{proof}

When $b=0$, namely when $\rho_{2}=\frac12$ the above ${}_{3}F_{2}$-series can
be evaluated. We have \cite[Theorem 3.5.5(ii)]{AAR}

\begin{lemm+}
\label{aar-lemma} Suppose $a_{1}+a_{2}=1, b_{1}+b_{2}=2a_{3}+1$. Then
\[
\,F\!\left(
\genfrac{}{}{0pt}{}{a_1, a_2, a_3}{b_1, b_2}%
\right)  =\frac{\pi\Gamma(b_{1})\Gamma(b_{2})} { 2^{2a_{3}-1} \Gamma
(\frac{a_{1}+b_{1}}2) \Gamma
\left
(\frac{a_{1}+b_{2}}2
\right
)
 \Gamma
\left
(\frac{a_{2}+b_{1}}2
\right
)
\Gamma
\left
(\frac{a_{2}+b_{2}}2
\right) }%
\]

\end{lemm+}

\begin{theo+}
\label{pos-dom} Let $(\mathfrak{g}_{0}, \mathfrak{k}_{0})$ be the symmetric
pair $(\mathfrak{su}(2, 2), \mathfrak{u}(2)+ \mathfrak{u}(2))$. Then
$\mathcal{A}_{0}=T_{1} \cup T_{2} $ is a union of two triangles, $T_{1}=[0,
\rho_{2}]^{2}\cap\mathcal{C}$, and
\[
T_{2}=\{(x_{1}, x_{2}), x_{1}\ge x_{2}\ge\rho_{2}, \quad x_{1}+x_{2}\le
\rho_{1}+\rho_{2}=2\}.
\]

\end{theo+}

\begin{proof}
The polynomial $q_{1, 0}(x)=-x_{1}^{2}-x_{2}^{2}+\rho_{1}^{2} +\rho_{2}^{2}$
and $q_{1, 1}(x)$ is by Lemma \ref{p-f-r=2} 
the
polynomial
\[
(\rho_{2}^{2}-x_{1}^{2})(\rho_{2}^{2}-x_{2}^{2}).
\]
The nonnegativity of $q_{1, 0}(x)$ is equivalent to $x_{1}^{2}+x_{2}^{2}%
\le\rho_{1}^{2}+\rho_{2}^{2}$ whereas that of $q_{1, 1}(x)$ is $x_{1},
x_{2}\le\rho_{2}$ or $x_{1}, x_{2}\ge\rho_{2}$.

The function $R(x)$ can now be evaluated by Lemma \ref{aar-lemma}, viz,
\[%
\begin{split}
R(x)  &  = \,F\!\left(
\genfrac{}{}{0pt}{}{\rho_{2}\pm x_{2}, \frac d2}{\rho_{1}\pm x_{1}}%
\right) \\
&  =\frac{\pi\Gamma(\rho_{1} +x_{1})\Gamma(\rho_{1} -x_{1})} {2^{d-1}
\Gamma
\left
(\frac{\rho_{1}+\rho_{2}+x_{1}+x_{2}}2
\right
)
 \Gamma
\left
(
\frac{\rho_{1}+\rho_{2}
+x_{1}-x_{2} }2
\right
) \Gamma
\left
(
\frac{\rho_{1}+\rho_{2} -x_{1}+x_{2} }2
\right
)
 \Gamma
\left(
\frac{\rho_{1}+\rho_{2}-x_{1}-x_{2}}2
\right)}.
\end{split}
\]
From which we see that $R(x)\ge0$ for $0\le x_{1}, x_{2}\le\rho_{2}$, and
$R(x)\ge0$ for $\rho_{2}\le x_{1}, x_{2}\le\rho_{1}$ if and only if
\[
x_{1}+x_{2}\le\rho_{1}+\rho_{2}.
\]
Our claim then follows from Theorem \ref{thm4.4}.
\end{proof}

\begin{rema+}
If $b>0$ the triangle $T_{2}=\{(x_{1}, x_{2}), x_{1}\ge x_{2}\ge\rho_{2},
\quad x_{1}+x_{2}\le\rho_{1}+\rho_{2}=2+b\}$ is not in the positivity domain
$\mathcal{A}_{0}$. Indeed if we put $x_{1}=x_{2}=\frac{\rho_{1}+\rho_{2}}2$,
then $(x_{1}, x_{2})\in T$ and the function $R(x)$ is
\[%
\begin{split}
R(x)  &  = \sum_{k=0}^{\infty} \frac{(2\rho_{2} +\frac12)_{k}(-\frac12)_{k} }
{(2\rho_{2} +\frac32)_{k}(\frac12)_{k} }\\
&  =1+\sum_{j=0}^{\infty} \frac{(2\rho_{2} +\frac12)_{j+1}(-\frac12)_{j+1} }
{(2\rho_{2} +\frac32)_{j+1}(\frac12)_{j+1} }\\
&  = 1+\sum_{j=0}^{\infty} \frac{(2\rho_{2} +\frac12)(-\frac12)} {(2\rho_{2}
+\frac32+j)(\frac12+j) }%
\end{split}
\]
by cancelling the common factors in the Pochammer symbols. This sum then can
be explicitly evaluated, viz
\[%
\begin{split}
R(x)  &  = 1+(2\rho_{2} +\frac12)(-\frac12) \sum_{j=0}^{\infty} \frac{1 }
{(2\rho_{2} +\frac32+j)(\frac12+j) }\\
&  = 1+(2\rho_{2} +\frac12)(-\frac12) \frac{1}{2\rho_{2} +1} \sum
_{j=0}^{\infty} (\frac{1 } {\frac12+j} - \frac{1 } {2\rho_{2} +\frac32+j})\\
&  = 1-\frac12(2\rho_{2} +\frac12) \frac{1}{2\rho_{2} +1} \sum_{k=0}%
^{2\rho_{2} } \frac{1}{\frac12 +k}%
\end{split}
\]
since it is a telescopic series. Now as a function of $2\rho_{2}= 1, 2,
\cdots$,
\[
\frac12(2\rho_{2} +\frac12) \frac{1}{2\rho_{2} +1} \sum_{k=0} ^{2\rho_{2} }
\frac{1}{\frac12 +k}
\]
attains its minimum $1$ when $2\rho_{2}=1$ namely when $b=0$, thus for $b>0$,
\[
R(x)= 1- \frac12(2\rho_{2} +\frac12) \frac{1}{2\rho_{2} +1} \sum_{k=0}%
^{2\rho_{2} } \frac{1}{\frac12 +k} <1-1=0.
\]
In the next section we shall give a different description of $\mathcal{A}_{0}$
and a different proof that the triangle $T_{2}$ is not in $\mathcal{A}_{0}$.
\end{rema+}

\begin{rema+}
We note that the unitarity set $\mathcal{U}\cap\mathcal{C}$ is the parameter
set for the spherical complementary series of $G$ and it has been determined
for $U(2,N)$ by Knapp and Speh \cite{Kn-Sp}. Let $k$ be the largest positive
integer such that $k\leq\frac{b-1}{2}$. Then $\mathcal{U}\cap\mathcal{C}$ is
the union of the following sets

\begin{enumerate}
\item the triangle $\{x\in\mathbb{R}_{\ge0}^{2}; 0\le x_{1} +x_{2}\le1\} $;

\item the triangles bordered by $x_{1}-x_{2}\ge j$ and $x_{1}+x_{2}\le j+1$ in
the triangle $[0, \rho_{2}]^{2}\cap\mathcal{C}$, $j=1, \cdots, k$;

\item line segments $x_{1}-x_{2}=j$ in the triangle $[0, \rho_{2}]^{2}%
\cap\mathcal{C}$, $j=1, \cdots, k$.
\end{enumerate}

Thus in this case $\mathcal{U}_{0}$ is a proper subset of $\mathcal{A}_{0}$.
\end{rema+}

\section{Alternative approach to $U\left(  m+2,2\right)  \label{sec-Alt}$}

\subsection{Limit formula for Okounkov polynomials}

We will need the following beautiful and simple identity for the $\Gamma$-function.

\begin{lemma}
Suppose $a, b, c, d\in\mathbb{C}\setminus$ $\left\{  0,-1,-2,\ldots\right\}  $
satisfy $a+b=c+d$ then%
\begin{equation}
\prod_{n=0}^{\infty}\frac{\left(  n+a\right)  \left(  n+b\right)  }{\left(
n+c\right)  \left(  n+d\right)  }=\frac{\Gamma\left(  c\right)  \Gamma\left(
d\right)  }{\Gamma\left(  a\right)  \Gamma\left(  b\right)  }. \label{=4G}%
\end{equation}

\end{lemma}

\begin{proof}
We recall the Weierstrass formula for the $\Gamma$-function
\[
\frac{1}{\Gamma\left(  z\right)  }=ze^{\gamma z}\prod_{n=1}^{\infty}\left\{
\left(  1+z/n\right)  e^{-z/n}\right\}  ,\quad z\neq0,-1,-2,\ldots
\]
(see e.g. \cite[P. 236]{Whitt-Wat}). Using this the right side of (\ref{=4G})
becomes%
\[
\frac{ab}{cd}\prod_{n=1}^{\infty}\left\{  \frac{\left(  1+a/n\right)  \left(
1+b/n\right)  e^{-\left(  a+b\right)  /n}}{\left(  1+c/n\right)  \left(
n+d/n\right)  e^{-\left(  c+d\right)  /n}}\right\}
\]
After canceling $e^{-\left(  a+b\right)  /n}=e^{-\left(  c+d\right)  /n}$ we
get the left side of (\ref{=4G}).
\end{proof}

Let $\psi_{l}\left(  t\right)  =\left(  -1\right)  ^{k}p_{l}\left(  t\right)
=\prod_{n=0}^{l-1}\left[  \left(  n+\alpha\right)  ^{2}-t^{2}\right]  $ be the
rank $1$ Okounkov polynomial. We will show how to compute the limit of the
rescaled polynomial
\begin{equation}
r_{l}\left(  t\right)  =\frac{\psi_{l}\left(  t\right)  }{\psi_{l}\left(
0\right)  }=\prod_{n=0}^{l-1}\frac{\left(  n+\alpha\right)  ^{2}-t^{2}%
}{\left(  n+\alpha\right)  ^{2}},\quad r\left(  t\right)  =\lim_{l\rightarrow
\infty}r_{l}\left(  t\right)  . \label{=q-l}%
\end{equation}
We are mainly interested in the case $\alpha=\frac{m+1}{2}$ where $m$ is a
non-negative integer. In this case the limit can be expressed in terms of the
function
\begin{equation}
\,s\left(  t\right)  =\frac{\sin\pi t}{\left(  t+1\right)  \cdots\left(
t+m\right)  } \label{=smt}%
\end{equation}
with $s(t)=\sin\pi t$ for $m=0$.

\begin{proposition}
\label{Prop-q}If $\alpha=\frac{m+1}{2}$ where $m$ is a non-negative integer
then
\begin{equation}
r\left(  t+\alpha\right)  =-\frac{\Gamma\left(  \alpha\right)  ^{2}}{\pi
}s\left(  t\right)  \label{=qtm}%
\end{equation}

\end{proposition}

\begin{proof}
Applying (\ref{=4G}) to (\ref{=q-l}) we get%
\[
r\left(  t\right)  =\prod_{n=0}^{\infty}\frac{\left(  n+\alpha+t\right)
\left(  n+\alpha-t\right)  }{\left(  n+\alpha\right)  \left(  n+\alpha\right)
}=\frac{\Gamma\left(  \alpha\right)  ^{2}}{\Gamma\left(  \alpha+t\right)
\Gamma\left(  \alpha-t\right)  },
\]
for $\alpha\notin\left\{  0,-1,-2,\ldots\right\}  $. For $\alpha=\left(
m+1\right)  /2$ this gives
\[
r\left(  t+\alpha\right)  =\frac{\Gamma\left(  \alpha\right)  ^{2}}%
{\Gamma\left(  m+1+t\right)  \Gamma\left(  -t\right)  }=\left[  \frac
{\Gamma\left(  \alpha\right)  ^{2}}{\left(  t+1\right)  \cdots\left(
t+m\right)  }\right]  \frac{1}{\Gamma\left(  1+t\right)  \Gamma\left(
-t\right)  },
\]
and (\ref{=qtm}) now follows from the elementary identity $\Gamma\left(
t\right)  \Gamma\left(  1-t\right)  =-\pi/\sin\pi t.$
\end{proof}

\subsection{The Shimura sets for $U\left(  m+2,2\right)  $}

In this section we consider the real points of the Shimura sets for the rank
$2$ groups $U\left(  m+2,2\right)  $. (So the root multiplicity $2b$ is now
$2m$.) For this we fix as before
\[
\alpha=\frac{m+1}{2},
\]
and write $q_{\lambda}\left(  x\right)  $ for $q_{\lambda}\left(
x;1,\alpha\right)  $. As above  we restrict attention to the
Weyl chamber $\mathcal C$ in the first quadrant $\mathbb{R}_{+}^{2}$ and we define
\begin{align*}
\mathcal{G}_{0}  &  =\left\{  x\in\mathcal C\mid q_{\left(
1,0\right)  }\left(  x\right)  ,\text{ }q_{\left(  1,1\right)  }\left(
x\right)  \geq0\text{ }\right\} \\
\mathcal{A}_{0}  &  =\left\{  x\in\mathcal C\mid q_{\lambda}\left(
x\right)  \geq0\text{ for all }\lambda\right\}  .
\end{align*}

Our description of these set will involve the triangles%
\[
T_{1}=\left[  0,\alpha\right]  \times\left[  0,\alpha\right]  \cap\mathcal{C},
\quad T_{2}=\left[  \alpha,\alpha+1\right]  \times\left[  \alpha
,\alpha+1\right]  \cap\mathcal{C}, .
\]
For $\mathcal{G}_{0}$ we consider the following subset of $T_{2}$
\[
V=\left\{  x\in T_{2}:q_{1,0}\left(  x\right)  \geq0\right\}  .
\]

\begin{theorem}
\label{Th:Gm}We have $\mathcal{G}_{0}=T_{1}\cup V$.
\end{theorem}

\begin{proof}
For $x\in\mathbb{R}_{+}^{2}$ the inequalities $q_{\left(  1,1\right)  }\left(
x\right)  \geq0$ and $q_{\left(  1,0\right)  }\left(  x\right)  $ are
respectively
\[
\left(  \alpha^{2}-x_{1}^{2}\right)  \left(  \alpha^{2}-x_{2}^{2}\right)
\geq0,\quad x_{1}^{2}+x_{2}^{2}\leq\alpha^{2}+\left(  \alpha+1\right)  ^{2}%
\]
The $q_{\left(  1,1\right)  }$ inequality holds iff either (a) $x\in T$ or (b)
$x_{1},x_{2}\geq\alpha$. In case (a) the $q_{\left(  1,0\right)  }$ inequality
is automatic, in case (b) it forces $x\in T_{2}$. The result follows.
\end{proof}

For $\lambda=\left(  l+k,k\right)  $ we have
\[
P_{\lambda}\left(  x\right)  =\frac{1}{x_{1}^{2}-x_{2}^{2}}\operatorname{det}%
\begin{bmatrix}
p_{l+k+1}\left(  x_{1}\right)  & p_{l+k+1}\left(  x_{2}\right) \\
p_{k}\left(  x_{1}\right)  & p_{k}\left(  x_{2}\right)
\end{bmatrix}
,
\]
which gives%
\begin{gather}
q_{\lambda}\left(  x\right)  =\left(  -1\right)  ^{\left\vert \lambda
\right\vert }P_{\lambda}\left(  x\right)  =\frac{\psi_{l+1}^{k}\left(
x_{2}\right)  -\psi_{l+1}^{k}\left(  x_{1}\right)  }{x_{1}^{2}-x _{2}^{2}}%
\psi_{k}\left(  x_{1}\right)  \psi_{k}\left(  x_{2}\right) \label{=qlam}\\
\text{where }\psi_{l+1}^{k}\left(  t\right)  =\frac{\psi_{l+k+1}\left(
t\right)  }{\psi_{k}\left(  t\right)  }=\prod\nolimits_{i=0}^{l}\left[
\left(  i+k+\alpha\right)  ^{2}-t^{2}\right] \nonumber
\end{gather}

\begin{lemma}
\label{Lem:T} The inequality $q_{\lambda}\left(  x\right)  \geq0$ holds in the
following cases.

\begin{enumerate}
\item If $x\in T$ and $\lambda$ is arbitrary.

\item If $x\in T_{2}$ and $\lambda_{2}=k>0.$
\end{enumerate}
\end{lemma}

\begin{proof}
By continuity and symmetry it suffices to prove $q_{\lambda}\left(  x\right)
\geq0$ for $x$ satisfying the additional conditions
\begin{equation}
x_{1}>x_{2},\quad x_{1},x_{2}\notin\left\{  \alpha,\alpha+1\right\}  .
\label{=rest}%
\end{equation}

In this case we have
\[
x_{1}^{2}-x_{2}^{2}>0,\quad0<x_{2}<x_{1}<\alpha
\]
Now $0<t<\alpha$, $\psi_{k}\left(  t\right)  $ is positive and $\psi_{l+1}%
^{k}\left(  t\right)  $ is positive and decreasing. It follows that
\begin{equation}
\psi_{k}\left(  x_{1}\right)  \psi_{k}\left(  x_{2}\right)  >0\text{ and }%
\psi_{l+1}^{k}\left(  x_{2}\right)  -\psi_{l+1}^{k}\left(  x_{1}\right)  >0.
\label{=ab-pos}%
\end{equation}
Thus by (\ref{=qlam}) we have $q_{\lambda}\left(  x\right)  \geq0$.

Let $\lambda=\left(  l+k,k\right)  $ with $k\geq1,$ and suppose $x\in T_{2}$
satisfies the assumptions (\ref{=rest}). Then we have
\[
x_{1}^{2}-x_{2}^{2}>0,\quad\alpha<x_{2}<x_{1}<\alpha+1.
\]
For $\alpha<t<\alpha+1$ and $k\geq1$, $\psi_{k}\left(  t\right)  $ is negative
and $\psi_{l+1}^{k}\left(  t\right)  $ is positive and decreasing. Once again
(\ref{=ab-pos}) holds and so $q_{\lambda}\left(  x\right)  \geq0.$
\end{proof}

We now describe $\mathcal{A}_{0}$ and for this we recall the function
$s\left(  t\right)  =\frac{\sin\pi t}{\left(  t+1\right)  \cdots\left(
t+m\right)  }$ as in the previous section, and we let $S\left(  x,y\right)  $
denote its symmetrized divided difference%

\begin{equation}
S\left(  x_1, x_2\right)  =\frac{s\left(  x_1\right)  -s\left(  x_2\right)  }%
{x_1-x_2}\text{ for }x_1\neq x_2,\quad S\left(  x,x\right)  =s^{\prime}\left(
x\right)  , \label{=S}%
\end{equation}
and we put$\quad$%
\[
W=\left\{  x\in T_{2}:S\left(  x-\alpha\right)  \geq0\right\}
\]
Here and elsewhere $x-\alpha$ denotes the pair $\left(  x_{1}-\alpha,
x_{2}-\alpha\right)  $.

\begin{theorem}
\label{Th:Am}We have $\mathcal{A}_{0}=T_{1}\cup W.$
\end{theorem}

\begin{proof}
By Thereom \ref{Th:Gm} we know that
\[
\mathcal{A}_{0}\subseteq\mathcal{G}_{0}=T\cup V\subseteq T\cup T_{2}.
\]
By Lemma \ref{Lem:T} it remains only to prove that for $x\in T_{2}$
\begin{equation}
q_{l,0}\left(  x\right)  \geq0\text{ for all }l\iff S\left(  x -\alpha\right)
\geq0\text{ } \label{=equiv}%
\end{equation}

Let $x_{1}\geq x_{2}$. We divide the proof of
(\ref{=equiv}) into three cases.

Case 1: We first consider $x\in T_{2}$ satisfying%
\begin{equation}
\alpha+1>x_{1}\geq x_{2}>\alpha. \label{=xi2}%
\end{equation}

This implies that $-\psi_{l+1}\left(  x_{2}\right)  ,$ $s\left(  x_{2}%
-\alpha\right)  ,$ and $x_{1}+x_{2}$ are all $>0$, and we define.
\[
c_{l}\left(  x\right)  =\frac{q_{l,0}\left(  x\right)  }{-\psi_{l+1}\left(
x_{2}\right)  },\quad c\left(  x\right)  =\frac{S\left(  x-\alpha\right)
}{\left(  x_{1}+x_{2}\right)  \left(  s\left(  x _{2}-\alpha\right)  \right)
}%
\]

By positivity (\ref{=equiv}) is equivalent to the assertion%
\begin{equation}
c_{l}\left(  x\right)  \geq0\text{ for all }l\iff c\left(  x\right)  \geq0.
\label{=c-equiv}%
\end{equation}
We will prove a stronger statement, namely%
\begin{equation}
c_{l}\left(  x\right)  \text{ is a decreasing sequence with limit }c\left(
x\right)  \label{=c-decr}%
\end{equation}
By continuity it suffices to prove (\ref{=c-decr}) under the additional
assumption $x_{1}>x_{2},$ and we may consider then the simpler expressions%
\begin{align*}
b_{l}  &  =\left(  x_{1}^{2}-x_{2}^{2}\right)  c_{l}\left(  x\right)
+1=\frac{\psi_{l+1}\left(  x_{1}\right)  }{\psi_{l+1}\left(  x_{2}\right)  }\\
b  &  =\left(  x_{1}^{2}-x_{2}^{2}\right)  c\left(  x\right)  +1=\frac
{s\left(  x_{1}-\alpha\right)  }{s\left(  x_{2}-\alpha\right)  }\quad
\end{align*}
Then $b_{l}$ and $b$ are strictly positive and we have%
\[
\frac{b_{l+1}}{b_{l}}=\frac{\alpha+l+1-x_{1}}{\alpha+l+1-x_{2}}\leq1.
\]
Moreover by Proposition \ref{Prop-q} we have $b_{l}\rightarrow b$. Thus
$b_{l}$ is a decreasing sequence with limit $b.$ This implies (\ref{=c-decr})
and hence (\ref{=c-equiv}) and (\ref{=equiv}).

Case 2: We now suppose that $x_{2}=\alpha$, so that $x$ is of the form
$\left(  x_{1},\alpha\right)  $. We claim that we have
\[
q_{l,0}\left(  x\right)  \geq0\text{ for all }l\text{,\quad}S\left(
x-\alpha\right)  \geq0.
\]
By continuity it suffices to prove this for $x_{1}\neq\alpha$ in which case it
follows from the explicit formula
\[
q_{l,0}\left(  x\right)  =\frac{-q_{l}\left(  x_{1}\right)  }{x_{1}^{2}%
-\alpha^{2}},\text{\quad}S\left(  x-\alpha\right)  =\frac{s\left(
x_{1}-\alpha\right)  }{x_{1}-\alpha}%
\]
Thus both sides of (\ref{=equiv}) are true and hence equivalent.

Case 3: Finally suppose that $x_{1}=\alpha+1$, so that $x$ is of the form
$\left(  \alpha+1,x_{2}\right)  .$ By Case 2 we may further supppose that
$x_{2}>\alpha$. With these assumptions we have $q_{1,0}\left(  x\right)  <0$.
So the left side of (\ref{=equiv}) is false and we need only prove that%
\begin{equation}
S\left(  x-\alpha\right)  <0. \label{=Sneg}%
\end{equation}

If $x_{2}\neq\alpha+1$ this follows from the explicit formula
\[
S\left(  x-\alpha\right)  =\frac{-s\left(  x_{2}-\alpha\right)  }{\left(
\alpha+1\right)  -x_{2}}.
\]

If $x_{2}=\alpha+1$ then $x$ is the point $\left(  \alpha+1,\alpha+1\right)  $
and we have
\[
S\left(  x-\alpha\right)  =S\left(  1,1\right)  =s^{\prime}\left(  1\right)
.
\]
To compute this derivative we recall the formula%
\[
s\left(  t\right)  =\frac{\sin\left(  \pi t\right)  }{g\left(  t\right)
},\quad g\left(  t\right)  =\left(  t+1\right)  \cdots\left(  t+m\right)
\]
Thus we have $g\left(  1\right)  =\left(  m+1\right)  !$ and%
\[
s^{\prime}\left(  t\right)  =\frac{\left(  \pi\cos\pi t\right)  g\left(
t\right)  -\left(  \sin\pi t\right)  g^{\prime}\left(  t\right)  }{g\left(
t\right)  ^{2}},\quad s^{\prime}\left(  1\right)  =-\frac{\pi}{\left(
m+1\right)  !}%
\]
This proves (\ref{=Sneg}) and hence (\ref{=equiv}).

Cases 1, 2, 3 establish then (\ref{=equiv}) for $x\in T_{2}$.
\end{proof}

Now if $m=0$ then it is clear that the set $W$ is the triangle $T_{2}$ borded
by $x_{1}=\frac{1}{2},x_{2}=\frac{1}{2},x_{1}+x_{2}=2$ so this agrees with
Theorem \ref{thm4.4}.


\section{Appendix}

In this appendix we write $(x, y)$
instead of $(x_1, x_2)$. 
The set $W$ of Theorem \ref{Th:Am} is the $\left(  \alpha,\alpha\right)  $
translate of the region in the positive quadrant bounded by the coordinate
axes and the curve defined implictly by the equation $S\left(  x,y\right)  =0$.

We write $S_{m}\left(  x,y\right)  $ for $S\left(  x,y\right)  $ to indicate
its dependence on $m$, and we give the graph of $S_{m}\left(  x,y\right)  =0$
for $m=0,1,2,3.$%

\[%
{\parbox[b]{2.5097in}{\begin{center}
\fbox{\includegraphics[
natheight=2.509700in,
natwidth=2.509700in,
height=2.5097in,
width=2.5097in
]%
{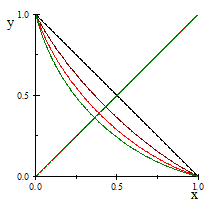}%
}\\
$S_m\left(  x,y\right)  =0,$ $m=0,1,2,3$
\end{center}}}%
\quad
\]
This graph is symmetric about the line $x=y$, and it is of some interest to
determine the point $c_{m}$ where the graph crosses the line $x=y.$

\begin{lemma}
The point $c=c_{m}$ satisfies the equation%
\begin{equation}
\pi\cot\pi c=%
{\textstyle\sum_{i=1}^{m}}
\frac{1}{\left(  c+i\right)  }. \label{=ceq}%
\end{equation}

\end{lemma}

\begin{proof}
It is easy to see that $c_{m}$ is a critical point of $s\left(  t\right)  $.
Since $s\left(  x\right)  $ is positive in the open interval $\left(
0,1\right)  ,$ its critical points are the same as those of the function%
\[
\ln\left(  s\left(  x\right)  \right)  =\ln\left(  \sin\pi x\right)  -%
{\textstyle\sum_{i=1}^{m}}
\ln\left(  x+i\right)  .
\]
This gives%
\[
\frac{d}{dx}\ln\left(  s\left(  x\right)  \right)  =\pi\cot\pi x-%
{\textstyle\sum_{i=1}^{m}}
\frac{1}{\left(  x+i\right)  }.
\]
The result follows by setting the derivative equal to $0.$
\end{proof}

\begin{corollary}
We have $c_{m}\rightarrow0$ as $m\rightarrow\infty$.
\end{corollary}

\begin{proof}
As $m\rightarrow\infty$ the right side of (\ref{=ceq}) approaches $\infty$ for
all $c$ in the interval $\left(  0,1\right)  $. Thus we must have $\pi
\cot\left(  \pi c_{m}\right)  \rightarrow\infty$ as well, which implies
$c_{m}\rightarrow0.$
\end{proof}

It seems likely that as $m\rightarrow\infty$ the region collapses to the union
of the unit intervals on the coordinate axes. However this requires an extra
convexity argument for the graph.

\providecommand{\bysame}{\leavevmode\hbox to3em{\hrulefill}\thinspace}
\providecommand{\MR}{\relax\ifhmode\unskip\space\fi MR }
\providecommand{\MRhref}[2]{  \href{http://www.ams.org/mathscinet-getitem?mr=#1}{#2}
} \providecommand{\href}[2]{#2}

\end{document}